\documentclass[a4paper,10pt]{amsart}

\usepackage{graphicx}

\usepackage{amssymb,amsmath,amsthm}
\usepackage{color}
\usepackage{mathrsfs} 
\usepackage{amssymb}
\usepackage{amsfonts}
\usepackage{bm}
\usepackage{fullpage}
\usepackage{color}
\usepackage{tikz}
\usepackage{framed}
\usepackage{graphicx}
\usepackage{url}
\usepackage{xcolor}
\usepackage{ marvosym }
\usepackage[skip=1em,font=small]{caption}
\usepackage{subcaption}
\usepackage[all]{xy}
\usepackage{listings}
\usepackage{multicol}

\newcommand\dto{\dashrightarrow}

\def\leq{\leqslant}
\def\geq{\geqslant}

\newcommand{\blue}[1]{#1}
\newcommand{\red}[1]{#1}

\newtheorem{thm}{Theorem}
\newtheorem{cor}[thm]{Corollary}
\newtheorem{lem}[thm]{Lemma}

\newtheorem{prop}[thm]{Proposition}

\newtheorem{defn}[thm]{Definition}
\newtheorem{nt}[thm]{Notation}
\newtheorem{rem}[thm]{Remark}

\newtheorem{exmp}{Example}

\DeclareMathOperator\Proj{Proj}

\newcommand\ZZ{\mathbb Z}
\newcommand\NN{\mathbb N}
\newcommand\MM{\mathbb M}
\newcommand\PP{\mathbb P}
\newcommand\TT{\mathrm T}

\newcommand\Sc{\mathcal S}
\newcommand\Rc{\mathcal R}

\newcommand\Zc{\mathcal Z}
\newcommand\Cc{\mathcal C}
\newcommand\Oc{\mathcal O}
\newcommand\Fk{\mathfrak F}
\newcommand\Lk{\mathfrak L}

\def\EE{\mathbb{E}}

\def\QQ{\mathbb{Q}}
\def\RR{\mathbb{R}}
\def\PP{\mathbb{P}}

\def\ux{{\xi}}

\def\dg{{\mathbf{d}}}
\def\mug{{\bm{\mu}}}
\def\nug{{\bm{\nu}}}
\def\ag{{\bm{a}}}

\def\mg{{\bm{m}}}

\def\Bc{{\mathcal{B}}}
\def\t{{\overline{t}}}
\def\u{{\overline{u}}}
\def\v{{\overline{v}}}

\def\pp{{\mathfrak{p}}}

\def\coker{{\mathrm{coker}}}

\def\corank{{\mathrm{corank}}}

\def\Rees{{\mathrm{Rees}}}
\def\Sym{{\mathrm{Sym}}}

\def\coker{{\mathrm{coker}}}

\def\ext{{\mathrm{Ext}}}

\def\corank{{\mathrm{corank}}}

\def\Rees{{\mathrm{Rees}}}
\def\Sym{{\mathrm{Sym}}}

\def\Supp{{\mathrm{Supp}}}

\def\CcB{{\breve{\mathscr{C}}}_B}
\def\ker{{\mathrm{ker}}}
\def\coker{{\mathrm{coker}}}

\def\CcB{{\breve{\mathscr{C}}}_B}
\def\ker{{\mathrm{ker}}}
\def\coker{{\mathrm{coker}}}

\newcommand\quotient[2]{
        \mathchoice
            {
                \text{\raise1ex\hbox{$#1$}\big/\lower1ex\hbox{$#2$}}%
            }
            {
                #1\,/\,#2
            }
            {
                #1\,/\,#2
            }
            {
                #1\,/\,#2
            }
    }

\begin{document}

\title[Fibers of rational maps and orthogonal projection onto surfaces]{Fibers of multi-graded rational maps and orthogonal projection onto rational surfaces}


\author{Nicol\'{a}s Botbol}
\address{Departamento de Matem\'atica, FCEN, Universidad de Buenos Aires, Argentina }
\email{nbotbol@dm.uba.ar}

\author{Laurent Bus\'e}
\address{Universit\'e C\^{o}te d'Azur, Inria,  2004 route des Lucioles, 06902 Sophia Antipolis, France}
\email{laurent.buse@inria.fr }

\author{Marc Chardin}
\address{Institut de Math\'ematiques de Jussieu and Universit\'e Sorbonne,  4 place Jussieu,
75005 Paris, France}
\email{marc.chardin@imj-prg.fr}

\author{Fatmanur Yildirim}
\address{Universit\'e C\^{o}te d'Azur, Inria,  2004 route des Lucioles, 06902 Sophia Antipolis, France}
\email{fatma.yildirim@inria.fr}

\begin{abstract}
We contribute a new algebraic method for computing the orthogonal projections of a point onto a rational algebraic surface embedded in \blue{the} three dimensional projective space. This problem is first turned into the computation of the finite fibers of a generically finite dominant rational map: a congruence of normal lines to the rational surface. Then, an in-depth study of certain syzygy modules associated to such a congruence is presented and applied to build elimination matrices that provide universal representations of its finite fibers, under some genericity assumptions. These matrices depend linearly in the variables of the three dimensional space. They can be pre-computed so that the orthogonal projections of points are approximately computed by means of fast and robust numerical linear algebra calculations.
\end{abstract}


\maketitle

\section{Introduction}\label{intro}

The computation of the distance between a 3D point and a parameterized algebraic surface has attracted a lot of interest, notably in the field of geometric modeling where B\'ezier surfaces, which are pieces of rational algebraic surfaces, are intensively used to describe geometric objects. This distance computation is very important in many applications such as surface construction, collision detection, simulation, or B-spline surface fitting (see for instance~\cite{NURBS,SJKW02,PoLe03}). It is commonly turned into a root-finding problem by considering the critical locus of the squared distance function from a test point $p$ to the parameterized surface. More precisely, if $p=(x_1,x_2,x_3)\in \RR^3$ and $\phi(u,v)=(\phi_1(u,v),\phi_2(u,v),\phi_3(u,v)) \in \RR^3$ is a  polynomial parameterization of a surface $\Sc$, then we need to find the common roots of the two partial derivatives with respect to $u$ and $v$ of the squared distance function
$$D_p(u,v):=\sum_{i=1}^3(\phi_i(u,v)-x_i)^2.$$
The roots of this polynomial system are called the orthogonal projections of the point $p$ onto the surface $\Sc$. Various types of practical algorithms to compute them have been proposed, most of them based on iterative subdivision methods, sometimes combined with Newton iterations for the computation of local extrema (see \cite{KS14} for a detailed survey). The number of roots of this polynomial system has also been recently studied with a larger spectrum of applications (after homogenization and counting properly multiplicities); it is known as the Euclidean distance degree of the surface $\Sc$ \cite{Draisma2015}. 

In this paper, we introduce a new approach, based on algebraic methods and elimination matrices, to compute the orthogonal projections of a point onto a rational algebraic surface in $\PP^3$. Starting with a bivariate parameterization $\phi$ of a rational surface $\Sc \subset \RR^3$, we consider a congruence of normal lines to $\Sc$, which is a trivariate parameterization. After homogenization of this congruence, denoted by $\Psi$, the syzygies of the ideal generated by the defining polynomials of $\Psi$ are used to build a family of matrices defined over $\PP^3$. These matrices have the property that their cokernels at a given point $p\in \PP^3$ are in close relation with the pre-images of $p$ via $\Psi$, hence with the orthogonal projections of $p$. Indeed, above a certain threshold, the dimension of the cokernels of these matrices evaluated at $p$ are equal to the number of pre-images of $p$ via $\Psi$, counted properly. Moreover the coordinates of those pre-images can be computed from a basis of these cokernels. 

\medskip

Algebraic methods to orthogonally project points onto rational algebraic surfaces already appeared in the literature \cite{La04,Draisma2015}, including by means of congruence of normal lines \cite{SJKW02,TJD05}, but they are facing computational efficiency problems very quickly as the defining degree of surface parameterizations is increasing, because of the intrinsic complexity of the problem. A good measure of this complexity is provided by the Euclidean distance degree introduced in \cite{Draisma2015} (see also \cite{JP14}); for instance, in general a point has 94 orthogonal projections onto a rational bi-cubic surface (a surface in $\PP^3$ parameterized over $\PP^1\times\PP^1$ by bi-homogeneous polynomials of bi-degree $(3,3)$). In order to push these limits, our approach introduces a preprocessing step in which an elimination matrix dedicated to a given rational surface, and depending linearly in the space coordinates, is built. The effective computation of the orthogonal projections of a point $p$ on this surface is then highly accelerated \blue{in comparison to other methods without preprocessing step (e.g.~\cite{TJD05}),} since it consists in the instantiation of this elimination matrix at $p$ and  the use of fast and robust numerical linear algebra methods, such as singular value decompositions, eigenvalue and eigenvector numerical calculations.

The methodology we develop in this paper is based on matrix representations of rational maps and their fibers. These representations have already been studied from a theoretical point of view in several contexts, see e.g.~\cite{BOTBOL2011381,BBC13,BD15,BDD,BL10,BJ,BDissac,BCD03},  and they have also been used for applications such as implicitization \cite{SC}, intersection problems and visualization via ray tracing \cite{Bus14}, isotropic triangular meshing of NURBS surfaces \cite{Shen:2016:LNS:3045889.3064444}, or more recently for isogeometric design and analysis of lattice-skin structures \cite{XiSaCi19,XiBuCi19}. \blue{Indeed, they provide an interesting alternative to the more classical resultant-based matrices that are very sensitive to the presence of base points and hence often inoperative, as for the computation of the orthogonal projections of a point on a rational surface we are considering in this paper.} Roughly, \blue{these representations} correspond to a presentation matrix of certain graded slices of the symmetric algebra of the ideal $I$ generated by the defining equations of the map under consideration. The determination of the appropriate graded slices is the main difficulty in this approach and it requires a thorough analysis of the syzygy modules of $I$. In this paper, guided by our application to orthogonal projection onto  rational surfaces, we consider trivariate maps whose source space is either equal to $\PP^2\times \PP^1$, a bi-graded algebraic structure, or $\PP^1\times \PP^1 \times \PP^1$, a tri-graded algebraic structure. In addition we also need to consider maps that have a one-dimensional base locus, because general congruences of normal lines to rational surfaces have positive dimensional base loci. This latter requirement is definitely the more challenging one. \blue{To the best of our knowledge}, all the previous published related works only considered rational maps with a base locus of dimension at most zero, i.e.~the base locus consists in finitely many points or is empty, with the exception of \cite{Cha06}. Our main results are  Theorem \ref{thm:regstabilization} and Theorem \ref{thm:mainresidual}. They provide the expected matrix representations under some assumptions on the curve component of the base locus, namely either being  globally generated by three linear combinations of the four defining equations of $I$  up to saturation, or either being a complete intersection. Coming back to our application, these theorems provide the theoretical foundations of a new methodology for computing orthogonal projections onto a rational algebraic surface (see e.g.~Example \ref{ex:segre}).

\medskip

The paper is organized as follows. In Section \ref{sec:surface} we introduce congruences of normal lines to a rational surface and explain why it is useful for computing orthogonal projections of points. Then, in Section \ref{sec:mrepsurface} the matrix representations of these congruences of normal lines are defined and the main results of this paper are stated. Their proofs require a fine analysis of the vanishing of certain local cohomology modules which is presented in Section \ref{sec:proofmainthm}.  In Section \ref{sec:secRing} we give some technical results on global sections of curves in a product of projective spaces in order to shed light on some assumptions that appear in  Theorem \ref{thm:regstabilization} and Theorem \ref{thm:mainresidual}. Finally, Section \ref{sec:Computation} is devoted to the description of an algorithm for computing the orthogonal projections of a point onto a rational surface which is parameterized by either $\PP^2$ (triangular surface) or $\PP^1\times\PP^1$ (tensor-product surface). 

\section{Congruence of normal lines to a rational surface}\label{sec:surface}

In this section, we introduce congruences of normal lines to a rational surface $\Sc$, i.e.~parameterizations of the 2-dimensional family of normal lines to $\Sc$. Given a point $p$ in space, it allows us to translate the computation of the orthogonal projections of $p$ onto $\Sc$ as the computation of the pre-images of $p$ via these congruences. In order to use algebraic methods, in particular elimination techniques, we first describe the homogenization of these congruence maps, as well as their base loci, for two classes of rational surfaces that are widely used in CAGD: triangular and tensor-product rational surfaces.  

\subsection{Congruences of normal lines}{}

\blue{We assume that we are given} the following affine parameterization of a rational surface $\Sc$ in the three dimensional space
\begin{equation}\label{eq:phiuvsurf}
\begin{array}{llll}
	\phi : & \RR^2 & \dto & \RR^3 \\
 	 & (u,v)  & \longmapsto & \left(\frac{f_1(u,v)}{f_0(u,v)}, \frac{f_2(u,v)}{f_0(u,v)},\frac{f_3(u,v)}{f_0(u,v)}\right)
\end{array}
\end{equation}
where $f_0,f_1,f_2,f_3$ are polynomials in the variables $u,v$. At each nonsingular point $p$ on $\Sc$ one can define a normal line which is the line through $p$ spanned by a normal vector $\nabla(p)$ to the tangent plane to $\Sc$ at $p$. The congruence of normal lines to $\Sc$ is then the rational map
\begin{equation}\label{eq:psiuvt}
\begin{array}{rcl}
	\psi: \RR^3 & \dto & \RR^3 \\
	(u,v,t) & \mapsto & \phi(u,v) +t\nabla(\phi(u,v)).
\end{array}	
\end{equation}
\blue{If the parameterization $\phi$ is given, then} there are many ways to formulate explicitly the above parameterization, depending on the choice of the expression of $\nabla(\phi(u,v))$. The more commonly used one is the cross product of the two  vectors $\partial {\phi}/\partial u$ and  $\partial {\phi}/\partial v$ that are linearly independent at almost all points in the image of $\phi$. Thus, we get 
$$\psi(u,v,t)=\phi(u,v)+t \cdot \frac{\partial \phi}{\partial u}\wedge \frac{\partial \phi}{\partial v}(u,v).$$
Nevertheless, with the following example we emphasize that, depending on $\phi$, a more specific and simple (for instance in terms of degree) expression of $\psi$ may be used. 

\begin{exmp} The unit sphere can be parameterized by 
	$$\phi(u,v)=\left(\frac{2u}{1+u^2+v^2},\frac{2v}{1+u^2+v^2},\frac{-1+u^2+v^2}{1+u^2+v^2}\right).$$ 
Since $\phi(u,v)$ is also a normal vector to the unit sphere at the point $\phi(u,v)$, a simpler expression than \eqref{eq:psiuvt} for the congruence of normal lines is  $\psi(u,v,t)=t\phi(u,v)$.\end{exmp}

\medskip

The main interest of the congruence of normal lines is that it allows to translate the computation of orthogonal projection onto the surface $\Sc$ as an inversion problem. More precisely, given a point $p \in \RR^3$, its orthogonal projection on $\Sc$ can be obtained from its pre-images via $\psi$. Indeed, if $\phi(u_0,v_0)$ is an orthogonal projection of $p$ on $\Sc$ for some parameters $(u_0,v_0) \in \RR^2$, then that means that $p$ belongs to the normal line to $\Sc$ at the point $\phi(u_0,v_0)$. Therefore, there exists $t_0$ such that the point $(u_0,v_0,t_0)\in \RR^3$ is a pre-image of $p$ via $\psi$.

\subsection{Homogenization to projective spaces}\label{subsec:homphi}
 
The geometric approach we propose for computing orthogonal projection of points onto a rational surface via the ``fibers'' of the corresponding congruence of normal lines relies on algebraic methods that require to work in an homogeneous setting. Thus, it is necessary to homogenize the parameterizations $\phi$ and $\psi$ defined by \eqref{eq:phiuvsurf} and \eqref{eq:psiuvt} respectively. 

\medskip

Regarding the homogenization of the map $\phi$, the canonical choice is to homogenize its source to the projective plane $\PP^2$, but there are other possible choices depending on the support of the polynomial $f_i$. We will focus on the two main classes of rational surfaces that are used in Computer-Aided Geometric Design (CAGD). The first class is called the class of  \textit{rational triangular surfaces}. It corresponds to polynomials of the form 
$$f_k(u,v)=\sum_{0\leq i,j, \ i+j\leq d} c_{k,i,j} u^iv^j, \ \ k=0,\ldots,3$$ 
where $d$ is a given positive integer, the degree of the polynomials $f_k(u,v)$. The canonical homogenization of the map $\phi$ is then of the form
$$\begin{array}{rcl}
\Phi : \PP^2 & \dto & \PP^3\\
(w:u:v) & \mapsto & (F_0:F_1:F_2:F_3)
\end{array}$$ 
where the $F_i$'s are homogeneous polynomials in $\RR[w,u,v]$ of degree $d$. If $F_0=w^d$, equivalently if $f_0(u,v)=1$ in \eqref{eq:phiuvsurf}, then one speaks of \emph{non-rational triangular surfaces}. This terminology refers to the fact that in this case that the parameterization $\phi$ is defined by polynomials and not by rational functions as in the general case.  

The second class of rational surfaces is called the class of \textit{rational tensor-product surfaces}. It corresponds to polynomials of the form 
$$f_k(u,v)=\sum_{0\leq i\leq d_1, \ 0\leq i\leq d_2} c_{k,i,j} u^iv^j, \ \ k=0,\ldots,3$$ 
where $(d_1,d_2)$ is a couple of positive integers, the bi-degree of the polynomials $f_k(u,v)$. In this case, the canonical homogenization of the map $\phi$ is of the form
$$\begin{array}{rcl}
\Phi : \PP^1\times \PP^1 & \dto & \PP^3\\
(\u:u)\times (\v:v) & \mapsto & (F_0:F_1:F_2:F_3)
\end{array}$$ 
where the $F_i$'s are here bi-homogeneous polynomials in $\RR[\u,u;\v,v]$ of bi-degree $(d_1,d_2)$. If $F_0=\u^{d_1}\v^{d_2}$, equivalently if $f_0(u,v)=1$ in \eqref{eq:phiuvsurf}, then one speaks of \emph{non-rational tensor-product surfaces}. 

\medskip

From now on we set the following notation. The map $\Phi$ is a rational map from the projective variety $X$ to $\PP^3$, where $X$ stands either for $\PP^2$ or $\PP^1\times \PP^1$. Thus, when we will use the notion of degree over $X$, it has to be understood with respect to these two possibilities, i.e.~either the single grading of $\PP^2$ or the bi-grading of $\PP^1\times\PP^1$. The homogeneous polynomials $F_0,F_1,F_2,F_3$ defining the map $\Phi$ are homogeneous polynomials in the coordinate ring of $X$ of degree $\dg$, this latter being either a positive integer or a pair of positive integers, depending on $X$. 

For the sake of completeness, in Table \ref{tab:EDdeg} we recall from \cite{Draisma2015}  the Euclidean distance degree of $\Phi$ corresponding to the four abovementioned classes of rational maps (see also \cite{JP14}). This is the number orthogonal projections of a general point not on the surface. It provides a measure of the complexity of computing the orthogonal projections of a point onto a rational surface.

\begin{table}[ht!]
	\centering	
	\renewcommand{\arraystretch}{1.5}
	\begin{tabular}{|c|c|c|}  
		\hline 
		&  Triangular surface &  Tensor-product surface \\ \hline
		Non-rational & $(2d-1)^2$ & $8d_1d_2-2(d_1+d_2)+1$  \\ \hline
		Rational & $7d^2-9d+3$ & $14d_1d_2-6(d_1+d_2)+4$ \\  \hline
	\end{tabular}
	\caption{Euclidean distance degree of non-rational and rational triangular surfaces, respectively tensor-product surfaces, of degree $d\geq1$, respectively $(d_1,d_2)\geq (1,1)$.}
	\label{tab:EDdeg}
\end{table}

\medskip

Finally, once the choice of homogenization of $\Phi$ from $X$ to $\PP^3$ is done, it is natural to homogenize $\psi$ as a rational map from $X\times \PP^1$ to $\PP^3$, which means geometrically that the congruence of normal lines to the surface is seen as a family of projective lines parameterized by $X$. This map is hence of the form 
\begin{eqnarray}\label{eq:Psi}
\Psi : X \times \PP^1 & \dto & \PP^3\\ \nonumber
\ux \times (\t:t) & \mapsto & (\Psi_0:\Psi_1:\Psi_2:\Psi_3)
\end{eqnarray} 
where the polynomials $\Psi_i$'s are bi-graded: they are graded with respect to $X$ and to $\PP^1$. Observe that these polynomials are actually linear forms with respect to $\PP^1$.

\subsection{Explicit homogeneous parameterizations}\label{subsec:homparam}

To describe the rational map $\Psi$ more explicitly, we need to consider projective tangent planes to the surface $\Sc$ and projective lines that are orthogonal to them. 

Let $\ux \in X$ and $p=\Phi(\ux)$ be a smooth point on $\Sc \subset \PP^3$. An equation of the projective tangent space to $\Sc$ at $p$, denoted $\TT_p \Sc$, can be obtained from the Jacobian matrix of the polynomials $F_0,F_1,F_2,F_3$ (see, for instance, \cite[\blue{Lecture} 14]{HarrisBook}). If $X=\PP^2$, then this equation is given by
$$
\left| 
\begin{array}{cccc}
\partial_u F_0 & \partial_u F_1 & \partial_u F_2 & \partial_u F_3  \\
\partial_v F_0 & \partial_v F_1 & \partial_v F_2 & \partial_v F_3  \\
\partial_w F_0 & \partial_w F_1 & \partial_w F_2 & \partial_w F_3  \\
x_0 & x_1 & x_2 & x_3
\end{array}
\right|
=
x_0\Delta_0(\ux)+x_1\Delta_1(\ux)+x_2\Delta_2(\ux)+x_3\Delta_3(\ux)=0.
$$
Observe that the signed minors $\Delta_i(w,u,v)$ are homogeneous polynomials of degree $D:=3(d-1)$, as the $F_i$'s are of degree $d$. Similarly, If $X=\PP^1\times \PP^1$, then an equation of $\TT_p \Sc$ is given by the vanishing of the determinant 
$$\left| 
\begin{array}{cccc}
	\partial_u F_0 & \partial_u F_1 & \partial_u F_2 & \partial_u F_3  \\
	\partial_\u F_0 & \partial_\u F_1 & \partial_\u F_2 & \partial_\u F_3  \\
	\partial_v F_0 & \partial_v F_1 & \partial_v F_2 & \partial_v F_3  \\
	x_0 & x_1 & x_2 & x_3
\end{array}
\right|=0.$$
Compared to the previous case, there is here a redundancy because two Euler equalities hold, one with respect to $(\u,u)$ and the other with respect to $(\v,v)$. Actually, this redundancy implies that the above determinant vanishes if $\v=0$. Therefore, an equation of $\TT_p \Sc$ is given by the formula
\begin{equation}\label{eq:Jactpsurf}
\left| 
\begin{array}{cccc}
\partial_u F_0 & \partial_u F_1 & \partial_u F_2 & \partial_u F_3  \\
\partial_\u F_0 & \partial_\u F_1 & \partial_\u F_2 & \partial_\u F_3  \\
\partial_v F_0 & \partial_v F_1 & \partial_v F_2 & \partial_v F_3  \\
x_0 & x_1 & x_2 & x_3
\end{array}
\right|
=
\v(x_0\Delta_0(\ux)+x_1\Delta_1(\ux)+x_2\Delta_2(\ux)+x_3\Delta_3(\ux))=0
\end{equation}
where the signed (and reduced) minors $\Delta_i(\u,u;\v,v)$ are bi-homogeneous polynomials of bi-degree $D:=(3d_1-2,3d_2-2)$.

Now, to characterize normal lines to $\Sc$ we will use the following property.  

\begin{lem} Let $H$ be a hyperplane in $\PP^3$ of equation $a_0x_0+a_1x_1+a_2x_2+a_3x_3=0$ and $L$ be a line in $\PP^3$, both not contained in the hyperplane at infinity $V(x_0) \subset \PP^3$. Then, $L$ is orthogonal to $H$, in the sense that their restrictions to the affine space $\PP^3 \setminus V(x_0)$ are orthogonal, if and only if the projective point $(0:a_1:a_2:a_3)$ belongs to $L$.
\end{lem}
\begin{proof} Let $H_1$, $H_2$ be two hyperplanes of equations $\sum_{i=0}^3 \alpha_i x_i=0$ and $\sum_{i=0}^3 \beta_ix_i=0$ respectively, and suppose that their intersection is exactly the line $L$. After restriction to the affine space $\PP^3 \setminus V(x_0)$, we have that the direction of $L$ is given by the cross product of the two vectors $(\alpha_1,\alpha_2,\alpha_3)$ and $(\beta_1,\beta_2,\beta_3)$. Therefore, we deduce that $L$ is orthogonal to $H$ if and only if the vector $(a_1,a_2,a_3)$ is orthogonal to both vectors   
$(\alpha_1,\alpha_2,\alpha_3)$ and $(\beta_1,\beta_2,\beta_3)$, which precisely means that the projective point $(0:a_1:a_2:a_3)$ belongs to the hyperplanes $H_1$ and $H_2$, hence to $L$.
\end{proof}

From the above property \blue{we deduce that there are two points that belong} to $\Sc$, namely the point $\Phi(\ux)=(F_0(\ux):F_1(\ux):F_2(\ux):F_3(\ux))$ and the point $(0:\Delta_1(\ux):\Delta_2(\ux):\Delta_3(\ux))$. Therefore, we can derive explicit rational parameterizations of the congruence of normal lines \eqref{eq:Psi} as follows.

If $X=\PP^2$ and $d\geq 2$, we get the following parameterization for the congruence of normal lines of a triangular rational surface
\begin{eqnarray}\label{eq:PsiHomP2}
\Psi : \PP^2 \times \PP^1 & \dto & \PP^3\\ \nonumber
(w:u:v) \times (\t:t) & \mapsto & (\Psi_0:\Psi_1:\Psi_2:\Psi_3)
\end{eqnarray} 
where
$$
\Psi_0 = \t w^{2d-3} F_0(w,u,v), \
\Psi_i = \t w^{2d-3} F_i(w,u,v) + t\Delta_i(w,u,v), \ i=1,2,3$$
are bi-homogeneous polynomials of bi-degree $(3d-3,1)$ over $\PP^2\times \PP^1$. 

If $X=\PP^1\times \PP^1$ and $d_1\geq 1$ and $d_2\geq 1$ we get the following parameterization for the congruence of normal lines of a tensor-product rational surface
\begin{eqnarray}\label{eq:PsiHomP1P1}
\Psi : \PP^1\times \PP^1 \times \PP^1 & \dto & \PP^3\\ \nonumber
(\u:u)\times (\v:v) \times (\t:t) & \mapsto & (\Psi_0:\Psi_1:\Psi_2:\Psi_3)
\end{eqnarray} 
where
\begin{align*}
 \Psi_0 &=\t \u^{2d_1-2}\v^{2d_2-2} F_0(\u,u;\v,v),\\
 \Psi_i&= \t  \u^{2d_1-2}\v^{2d_2-2} F_i(\u,u;\v,v) + t\Delta_i(\u,u;\v,v), \ i=1,2,3
\end{align*}
are tri-homogeneous polynomials of degree $(3d_1-2,3d_2-2,1)$ over $\PP^1\times \PP^1 \times \PP^1$. 

\medskip

We emphasize that the above parameterizations hold for general rational triangular and tensor-product surfaces, so that some simplifications may appear in some particular cases. For instance, such simplifications are obtained with non-rational triangular and tensor-product surfaces. Indeed, in those cases the polynomials $\Delta_1,\Delta_2$ and $\Delta_3$ have a common factor, namely $w^{d-1}$ if $X=\PP^2$, and $\u^{d_1-1}\v^{d_2-1}$ if $X=\PP^1\times \PP^1$. Therefore, this common factor propagates to the polynomials $\Psi_0,\ldots,\Psi_3$ and hence the corresponding  parameterization $\Psi$ of the congruence of normal lines is given by polynomials of bi-degree $(2d-2,1)$ if $X=\PP^2$ and of tri-degree $(2d_1-1,2d_2-1,1)$ if $X=\PP^1\times \PP^1$. These considerations are summarized in Table \ref{tab:degreePsi}.

\begin{table}[ht]
\centering	
\renewcommand{\arraystretch}{1.5}
\begin{tabular}{|c|c|c|}  
	\hline 
	$\deg(\Psi_i)$ &  Triangular surface &  Tensor-product surface \\ \hline
	Non-rational & $(2d-2,1)$ & $(2d_1-1,2d_2-1,1)$  \\ \hline 
	Rational & $(3d-3,1)$ & $(3d_1-2,3d_2-2,1)$ \\  \hline
\end{tabular}
\caption{Degree of the parameterizations of the congruence $\Psi$ of normal lines associated to  non-rational/rational triangular/tensor-product surfaces}
\label{tab:degreePsi}
\end{table}

\subsection{Base locus}\label{sec:baselocus} Consider the map $\Psi$ defined by \eqref{eq:Psi}. 
Its base locus $\Bc$ 
is the subscheme of $X\times \PP^1$ defined by the polynomials $\Psi_0,\Psi_1,\Psi_2,\Psi_3$. As we will see in the next section, this locus is of particular importance in our syzygy-based approach for studying the ``fibers'' of $\Psi$. 

Without loss of generality, $\Bc$ can be assumed to be of dimension at most one by simply removing the common factor of the polynomials $\Psi_0,\ldots,\Psi_3$, if any. It is clear from \eqref{eq:PsiHomP2} and \eqref{eq:PsiHomP1P1} that $\Bc$ is always one-dimensional. In the following lemma we describe the curve component of $\Bc$ when the corresponding surface parameterization $\Phi$ is sufficiently general. Below, inequalities between tuples of integers are understood component-wise.

\begin{lem}\label{lem:baselocus} For a general rational surface parameterization $\Phi$, the curve component of the base locus $\Bc$ of its corresponding congruence of normal lines $\Psi$ as defined in \eqref{eq:PsiHomP2} and \eqref{eq:PsiHomP1P1},  is given by 
	\begin{itemize}
		\item the ideal $(w^{2d-3},t)$ if $X=\PP^2$ and $d\geq 2$,
		\item the ideal $(\u^{2d_1-2}\v^{2d_2-2},t)$ if $X=\PP^1\times \PP^1\times \PP^1$ and $(d_1,d_2)\geq (1,1)$.
	\end{itemize}
Similarly, for a general non-rational surface parameterization $\Phi$, the base locus $\Bc$ of $\Psi$ is a one-dimensional subscheme of $X\times \PP^1$ whose curve component is defined by 
	\begin{itemize}
		\item the ideal $(w^{d-2},t)$ if $X=\PP^2$ and $d\geq 2$,
		\item the ideal $(\u^{d_1-1}\v^{d_2-1},t)$ if $X=\PP^1\times \PP^1\times \PP^1$ and $(d_1,d_2)\geq(1,1)$.
	\end{itemize}
\end{lem}
\begin{proof} We only consider the case $X=\PP^2$ and $\Phi$ rational, the other cases are similar. By \eqref{eq:PsiHomP2}, we have the matrix equality
$$ \left(
\begin{array}{cccc}
	\Psi_0 & \Psi_1 & \Psi_2 & \Psi_3
\end{array}
\right)=
\left(
\begin{array}{cc}
	\t w^{2d-3} & t 
\end{array}
\right)\cdot 
\left(\begin{array}{cccc}
	F_0 & F_1 & F_2 & F_3 \\
	0 & \Delta_1 & \Delta_2 & \Delta_3
\end{array}
\right)
$$
so that the ideal $(\Psi_0,\ldots,\Psi_3):(\t w^{2d-3},t )$ is contained in the ideal generated by $\Psi_0,\ldots, \Psi_3$ and the 2-minors of the matrix 
$$\mathbb{F}:=\left(\begin{array}{cccc}
	F_0 & F_1 & F_2 & F_3 \\
	0 & \Delta_1 & \Delta_2 & \Delta_3
\end{array}
\right).$$
Then \blue{the first row of $\mathbb{F}$ vanishes at the base points} of the surface parameterization $\Phi$, which are assumed to be finitely many. The second row of $\mathbb{F}$ vanishes at the singular points of the image of $\Phi$ and at the points where the tangent plane is of equation $x_0=0$; if $\Phi$ is sufficiently general then these latter points are also finitely many. Finally, the two rows of $\mathbb{F}$ are proportional at finitely many points such that $F_0=0$, always assuming that $\Phi$ is sufficiently general. Therefore, we deduce that the curve component of the ideal defined by the $\Psi_i$'s is defined by the ideal $(w^{2d-3},t)$ in $\PP^2\times \PP^{1}$. 
\end{proof}

\begin{rem} 
If $X=\PP^1\times\PP^1$ and $(d_1,d_2)=(1,1)$ then there is no curve component in the base locus $\Bc$ for both non-rational and rational surface parameterizations. The same holds if $X=\PP^2$ and $d=2$ for non-rational surface parameterizations. 
\end{rem}

\section{Fibers and matrices of syzygies}\label{sec:mrepsurface}

In this section, we extend our framework and \blue{we assume that we are given a} homogeneous parameterization
\begin{eqnarray}\label{eq:defPsi}
	\Psi: X\times\PP^1 & \dto & \PP^3 \\ \nonumber
	       \ux \times(\t:t) & \mapsto &
		   \left(\Psi_0: \Psi_1: \Psi_2:  \Psi_3  \right),
\end{eqnarray}
where $X$ stands for the spaces $\PP^2$ or $\PP^1\times\PP^1$ over an algebraically closed field $k$, and the $\Psi_i$'s are homogeneous polynomials in the coordinate ring of $X\times \PP^1$ for all $i=0,1,2,3$. The coordinate ring $R_X$ of $X$ is equal to $k[w,u,v]$ or $k[\u,u;\v,v]$, respectively, depending on $X$. The coordinate ring of $\PP^1$ is denoted by $R_1=k[\t,t]$ and hence the coordinate ring of $X\times \PP^1$ is the polynomial ring 
$$R:=R_X\otimes_k R_1.$$
Thus, the polynomials $\Psi_0,\Psi_1,\Psi_2,\Psi_3$ are multi-homogeneous polynomials of degree $(\dg,e)$, where $\dg$ refers to the degree with respect to $X$, which can be either an integer if $X=\PP^2$, or either a couple of integers if $X=\PP^1\times\PP^1$. 

\medskip

In what follows, we assume that $\Phi$ is a dominant map. We denote by $I$ the ideal generated by $\Psi_0,\Psi_1,\Psi_2,\Psi_3$ in $R$. The base locus $\Bc$ of the map $\Psi$ is the subscheme of $X\times \PP^1$ defined by $I$. Without loss of generality, we assume that $\Bc$ is of dimension at most one. Our aim is to provide a matrix-based representation of the finite fibers of $\Psi$, by means of the syzygies of the ideal $I$. 

\subsection{Fiber of a point}
The map $\Psi$ being a rational map, its fibers are not well defined. To give a proper definition of the fiber of a point under $\Psi$, we need to consider the graph of $\Psi$ and its closure $\Gamma \subset X\times \PP^1\times \PP^3$. The defining equations of $\Gamma$ are the equations of the multi-graded Rees algebra of the ideal $I$ of $R$, denoted $\Rees_R(I)$, which is a domain. It has two canonical projections $\pi_1$ and $\pi_2$ onto $X\times \PP^1$ and $\PP^3$ respectively: 

\begin{equation}\label{eq:diagramGraph}
\xymatrix@1{\Gamma\ \ar[d]_{\pi_1}\ar[rrd]^{\pi_2}\ar@{^(->}[rr] & & X\times \PP^1 \times \PP^3 \\  X\times \PP^1 \ar@{-->}[rr]^{\Psi} &  & \PP^3}	
\end{equation}
Thus, the fiber of a point $p\in \PP^3$ is defined through the regular map $\pi_2$, i.e.~as $\pi_2^{-1}(p)$. More precisely, if $\kappa(p)$ denotes the residue field of $p$, \textit{the fiber of $p\in \PP^3$} is the subscheme 
\begin{equation}\label{eq:fiberRees}
	\Fk_p:=\Proj(\Rees_R(I)\otimes \kappa(p)) \subset X\times \PP^1.
\end{equation}
\blue{(We refer the reader to \cite[\S 6.5]{Eis95} for a standard reference on Rees algebras and to \cite[\S 3]{Cha06}, and the references therein, for more details on geometric considerations.)}

From a computational point of view, the equations of the Rees algebra $\Rees_R(I)$ are very difficult to get. Therefore, it is useful to approximate the Rees algebra $\Rees_R(I)$ by the corresponding symmetric algebra of the ideal $I$, that we denote by $\Sym_R(I)$. This approximation amounts to keep among defining equations of the Rees algebra only those that are \textit{linear} with respect to the third factor $\PP^3$. Thus, as a variation of the standard definition \eqref{eq:fiberRees} of the fiber of a point $p\in \PP^3$, we will consider the subscheme 
\begin{equation}\label{eq:linearfiber}
\Lk_p:=\Proj(\Sym_R(I)\otimes \kappa(p)) \subset X\times \PP^1
\end{equation}
that we call \textit{the linear fiber of $p$}. We emphasize that the fiber $\Fk_p$ is always contained in the linear fiber $\Lk_p$ of a point $p$, and that they coincide if the ideal $I$ is locally a complete intersection at $p$.

\subsection{Matrices built from syzygies}\label{subsec:matrixrep}

Given a point $p$ in $\PP^3$, the dimension and the degree of its linear fiber $\Lk_p$ can be read from its Hilbert polynomial. The evaluation of this polynomial, and more generally of its corresponding Hilbert function, can be done by computing the rank of a collection of matrices that we will introduce. They are built from the syzygies of the ideal $I$. An additional motivation to consider these matrices is that they also allow to compute effectively the points defined by the linear fiber $\Lk_p$, hence the pre-images of $p$ via $\Psi$, if $\Lk_p$ is finite. This property will be detailed in Section \ref{sec:Computation}.

\medskip

Let $k[x_0,x_1,x_2,x_3]$ be the coordinate ring of $\PP^3$. The symmetric algebra $\Sym_R(I)$ of the ideal $I=(\Psi_0,\ldots,\Psi_3)$ of $R$ is the quotient of the polynomial ring $R[x_0,x_1,x_2,x_3]$ by the ideal generated by the linear forms in the $x_i$'s whose coefficients are syzygies of the polynomials $\Psi_0,\ldots,\Psi_3$. More precisely,
consider the graded map 
\begin{eqnarray*}
	R(-\dg,-e)^4 & \rightarrow & R \\
	(g_0,g_1,g_2,g_3) & \mapsto & \sum_{i=0}^3 g_i\Psi_i 
\end{eqnarray*}
and denote its kernel by $Z_1$, which is nothing but the first module of syzygies of $I$. Setting $\Zc_1:=Z_1(\dg,e)\otimes R[x_0,\ldots,x_3]$ and $\Zc_0=R[x_0,\ldots,x_3]$, then the symmetric algebra $\Sym_R(I)$ admits the following multi-graded presentation
\begin{eqnarray}\label{eq:defvarphi}
	\Zc_1(-1) & \xrightarrow{\varphi} & \Zc_0 \rightarrow \Sym_R(I) \rightarrow 0 \\ \nonumber
	 	(g_0,g_1,g_2,g_3) & \mapsto & \sum_{i=0}^3 g_ix_i.
\end{eqnarray}
where the shift in the grading of $\Zc_1$ is with respect to the grading of $k[x_0,\ldots,x_3]$.

\begin{defn}\label{def:mrep} The graded component of the mapping $\varphi$ in any degree $(\mug,\nu)$ with respect to $R=R_X\otimes R_1$ gives a map of free $k[x_0,\ldots,x_3]$-modules. Its matrix, which depends on a choice of basis, is denoted by $\MM_{(\mug,\nu)}(\Psi)$, or simply by $\MM_{(\mug,\nu)}$. Its entries are linear forms in $k[x_0,\ldots,x_3]$.
\end{defn}

We notice that the computation of $\MM_{(\mug,\nu)}$ is rather simple from the parameterization $\Psi$. Indeed it amounts to compute a basis of the $k$-vector space of syzygies of degree $(\mug,\nu)$ of $I$, which can be done by solving a single linear system. As an illustration, here is the {\sc Macaulay2} \cite{M2} code for building such a matrix under the form $\MM_{(\mug,\nu)}=\sum_{i=0}^{3}x_i\MM_i$, the $\MM_i$'s being matrices with entries in $k$, 
 starting from a rational triangular surface parameterization $\Phi$. 
 
\begin{verbatim}
R=QQ[u,v,w]
d=2; f=random(R^{d},R^{0,0,0,0}) -- the surface parameterization
JM=jacobian f;
J=ideal(-determinant(JM_{1,2,3}), determinant(JM_{0,2,3}), 
  -determinant(JM_{0,1,3}),determinant(JM_{0,1,2}));
Sg=QQ[u,v,w,t,tt,Degrees=>{{1,0},{1,0},{1,0},{0,1},{0,1}}]
fg=sub(f,Sg);Jg=sub(J,Sg);
P0=tt*w^(2*d-3)*fg_(0,0)
P1=tt*w^(2*d-3)*fg_(0,1)+t*Jg_1
P2=tt*w^(2*d-3)*fg_(0,2)+t*Jg_2
P3=tt*w^(2*d-3)*fg_(0,3)+t*Jg_3 -- Pi's : congruence of normal lines 
D=3*d-3;E=1; -- degree of the Pi's
mu=4;nu=1; -- a particular matrix
Breg=basis({mu,nu},Sg); br=rank source Breg;
SM=Breg*P0|Breg*P1|Breg*P2|Breg*P3;
(sr,Sreg)=coefficients(SM,Variables=>{u_Sg,v_Sg,w_Sg,t_Sg,tt_Sg},
          Monomials=>basis({mu+D,nu+E},Sg));
MM=syz Sreg; 
M0=MM^{0..br-1};M1=MM^{br..2*br-1};
M2=MM^{2*br..3*br-1};M3=MM^{3*br..4*br-1};
\end{verbatim}	

\medskip

As a consequence of \eqref{eq:defvarphi}, for any point $p\in \PP^3$ the corank of the matrix $\MM_{(\mug,\nu)}$ evaluated at $p$, that we denote by $\MM_{(\mug,\nu)}(p)$, is equal to the Hilbert function of the linear fiber $\Lk_p$ in degree $(\mug,\nu)$. Because the Hilbert function is equal to its corresponding Hilbert polynomial for suitable degrees $(\mug,\nu)$, the corank of the matrix $\MM_{(\mug,\nu)}(p)$ is expected to stabilize for large values of $\mug$ and $\nu$ to a constant value if $\Lk_p$ is finite. In the next section we provide effective bounds for this stability property under suitable assumptions. 

{
\begin{exmp}\label{ex:segreMrep} 
	Consider the Segre embedding corresponding to bilinear polynomials, i.e.
	$$\begin{array}{rcl}
	\Phi : \PP^1\times \PP^1 & \red{\to} & \PP^3\\
	(\u:u)\times (\v:v) & \mapsto & (\u\v:u \v: \u v:uv).
	\end{array}$$
	Then, a parameterization $\Psi$ of the congruence of normal lines to \red{the Segre surface (image of $\Phi$)} is  \begin{eqnarray}\label{eq:paramCongruenceSegre}
	\Psi : \PP^1\times\PP^1\times\PP^1 & \rightarrow & \PP^3\nonumber \cr
	(\bar{u}:u)\times (\bar{v}:v)\times (\bar{t},t) & \mapsto & (\bar{t}\bar{u}\bar{v} : \bar{t}u\bar{v}+t\bar{u}v : \bar{t}v\bar{u} +t\bar{u}\bar{v} : \bar{t}uv -t\bar{u}\bar{v})
	\end{eqnarray} of degree $(1,1,1)$ on $\PP^1\times\PP^1\times\PP^1$. Choosing the integer $(\mug,\nu)=(2,2,0)$ (see   \textsection{\ref{subsec:MrepOfLinearFibers}} for an explanation of this particular choice), the matrix $\MM_{(\mug,\nu)}(\Psi)$ is easily computed with {\sc Macaulay2} \cite{M2}:   
	{\small 
		$$\MM_{(2,2,0)}:=\begin{pmatrix}
		0&{-{x}_{0}}&0&0&0\\
		0&0&{-{x}_{0}}&{x}_{1}&0\\
		{x}_{0}&{-{x}_{0}}&0&{-{x}_{3}}&0\\
		0&0&0&{-{x}_{2}}&{-{x}_{0}}\\
		0&{x}_{3}&0&0&0\\
		{-{x}_{1}}&{x}_{1}&{x}_{3}&{-{x}_{2}}&{-{x}_{0}}\\
		{-{x}_{0}}&0&0&{x}_{3}&0\\
		{x}_{2}&0&{-{x}_{0}}&{x}_{1}&{x}_{3}\\
		0&0&{x}_{2}&0&{x}_{1}\end{pmatrix}.$$
	}
	
	\noindent Its rank at a general point in $\PP^3$ is equal to 4, so that the dimension of its cokernel is equal to 5, which is nothing but the Euclidean distance degree of \red{the Segre surface in $\PP^3$}.	
\end{exmp}
}

\subsection{The main results}\label{subsec:mainres} We recall that $I$ is the ideal of $R=R_X\otimes R_1$ generated by the defining polynomials of the map $\Psi$, i.e. $I:=(\Psi_0,\Psi_1,\Psi_2,\Psi_3)$. The irrelevant ideal of $X\times\PP^1$ is denoted by $B$; it is equal to the product of ideals $(w,u,v)\cdot (\t,t)$ if $X=\PP^2$, or to the product $(\u,u)\cdot (\v,v)\cdot (\t,t)$  if $X=\PP^1\times \PP^1$. The notation $I^{\textrm{sat}}$ stands for the saturation of the ideal $I$ with respect to the ideal $B$, i.e.~$I^{\textrm{sat}}=(I:B^\infty)$. The homogeneous polynomials $\Psi_0,\Psi_1,\Psi_2,\Psi_3$ are of degree $(\dg,e)$, where $\dg$ denotes the degree with respect to $X$ and $e$ denotes the degree with respect to $\PP^1$. We recall that inequalities between tuples of integers are understood component-wise. 

\medskip

The base locus of $\Psi$ is the subscheme of $X\times \PP^1$ defined by the ideal $I$; it is denoted by $\mathcal{B}$. Without loss of generality, $\Bc$ is assumed to be of dimension at most one; if $\dim(\Bc)=1$ then we denote by $\Cc$ its top unmixed one-dimensional curve component. 

\begin{defn}\label{def:negsec} The curve $\Cc\subset X \times \PP^1$ has no section in degree $<(\ag,b)$ if for any $\bm{\alpha}<\ag$ and $\beta<b$, 
$H^0(\Cc,\mathcal{O}_\Cc(\bm{\alpha},\beta))=0$. 
\end{defn}

\blue{We are now ready to state our main results, but for the sake of simplicity we first need to introduce the following notation.}

\begin{nt}\label{not:E} Let $r$ be a positive integer. For any $\bm{\alpha}=(\alpha_1,\ldots,\alpha_r) \in (\ZZ \cup \{ -\infty\} )^r$ we set 
	$$\EE (\bm{\alpha }):=\{ \bm{\zeta} \in \ZZ^r \ \vert \ \zeta_i \geq \alpha_i \textrm{ for all } i=1,\ldots,r\}.$$ 
It follows that for any $\bm{\alpha}$ and $\bm{\beta}$ in $(\ZZ \cup \{ -\infty\} )^r$ we have that $\EE (\bm{\alpha})\cap \EE (\bm{\beta})=\EE (\bm{\gamma})$ where $\gamma_i =\max\{ \alpha_i ,\beta_i \}$ for all $i=1,\ldots,r$, i.e.~ $\bm{\gamma}$ is the maximum of $\bm{\alpha}$ and $\bm{\beta}$ component-wise.	
\end{nt}

\begin{thm}\label{thm:regstabilization} Assume that we are in one of the two following cases:
	\begin{itemize}
		\item[(a)] The base locus $\Bc$ is finite, possibly empty,
		\item[(b)] $\dim(\Bc)=1$, $\Cc$ has no section in degree $<(\bm{0},e)$ and $I^{\textrm{sat}}=I'^{\textrm{sat}}$ where $I'$ is an ideal generated by three general linear combinations of the polynomials $\Psi_0,\ldots,\Psi_3$.
  	\end{itemize}
	Let $p$ be a point in $\PP^3$ such that $\Lk_p$ is finite, then
	$$ \corank \, \MM_{(\mug,\nu)}(p) =\deg(\Lk_p)$$
	for any degree $(\mug,\nu)$ such that
	\begin{itemize}
		\item if $X=\PP^2$,
		\begin{equation}\label{eq:thmain1-P2}
			(\mu,\nu)\in \EE(3d-2,e-1) \cup \EE(2d-2,3e-1).
		\end{equation}
		\item if $X=\PP^1\times \PP^1$, 
		\begin{multline}\label{eq:thmain1-P1P1}
			(\mug,\nu) \in \EE(3d_1-1,2d_2-1,e-1) \cup \EE(2d_1-1,3d_2-1,e-1) \\
			\cup \EE(2d_1-1,2d_2-1,3e-1).
		\end{multline}
	\end{itemize}	
\end{thm}
\begin{proof} By \eqref{eq:defvarphi}, the corank of the matrix $\MM_{(\mug,\nu)}(p)$ is equal to the Hilbert function of $\Sym_R(I)\otimes\kappa(p)$ at $(\mug,\nu)$, that we denote by $HF_{\Sym_R(I)\otimes\kappa(p)}(\mug,\nu)$. Moreover, since $\Lk_p$ is assumed to be finite, the Hilbert polynomial of $\Sym_R(I)\otimes\kappa(p)$, denoted $HP_{\Sym_R(I)\otimes\kappa(p)}(\mug,\nu)$, is a constant polynomial which is equal to the degree of $\Lk_p$. 
 Now, the Grothendieck-Serre formula shows that for any degree $(\mug,\nu)$ we have the equality (see for instance \cite[Proposition 4.26]{BotCh})
$$ HP_{\Sym_R(I)\otimes\kappa(p)}(\mug,\nu)=HF_{\Sym_R(I)\otimes\kappa(p)}(\mug,\nu) - \sum_{i\geq 0} (-1)^iHF_{H^i_B(\Sym_R(I)\otimes\kappa(p))}(\mug,\nu)$$
\blue{where $H^i_B(-)$ stands for local cohomology modules (we refer the reader to \cite[\S 4]{Cha06} for a brief introduction to these modules in a similar setting).}
Therefore, the theorem will be proved if we show that the Hilbert functions of the local cohomology modules $H^i_B(\Sym_R(I)\otimes\kappa(p))$ vanish for all integers $i$ and all degrees $(\mug,\nu)$ satisfying to the conditions stated in the theorem. This property is the content of Theorem \ref{thm:main} whose proof is postponed to Section \ref{sec:proofmainthm}.
\end{proof}

In the case where the base locus $\Bc$ has dimension one, the assumption $I^{\textrm{sat}}=I'^{\textrm{sat}}$ in Theorem \ref{thm:regstabilization}, item (b), can be a restrictive requirement, in particular in our targeted application for computing orthogonal projections onto rational surfaces. The next result allows us to remove this assumption in the case the curve component $\Cc$ of the base locus is a complete intersection.

\begin{thm}\label{thm:mainresidual} Assume that $\dim(\Bc)=1$ and that $\Cc$ has no section in degree $<(\bm{0},e)$. Moreover, assume that there exists an homogeneous ideal $J\subset R$ generated by a regular sequence $(g_1,g_2)$ such that $I\subset J$ and $(I:J)$ defines a finite subscheme in $X\times \PP^1$. Denote by $(\mg_1,n_1)$, resp.~$(\mg_2,n_2)$, the degree of $g_1$, resp.~$g_2$, set $\eta:=\max(e-n_1-n_2,0)$ and let $p$ be a point in $\PP^3$ such that $\Lk_p$ is finite. Then, 
	$$ \corank \, \MM_{(\mug,\nu)}(p) =\deg(\Lk_p)$$
	for any degree $(\mug,\nu)$ such that
	\begin{itemize}
		\item if $X=\PP^2$
		\begin{equation}\label{eq:thmain2-P2}
			(\mu,\nu)\in \EE (3d-2,e-1+\eta) \cup \EE (2d-2+d-\min\{m_1,m_2\},3e-1).
		\end{equation}
		\item if $X=\PP^1\times \PP^1$
		\begin{multline}\label{eq:thmain2-P1P1}
			(\mug,\nu)\in \EE(3d_1-1,2d_2-1+\tau_2,e-1+\eta) \cup \\
			\EE(2d_1-1+\tau_1,3d_2-1,e-1+\eta)\cup \\ \EE(2d_1-1+\tau_1,2d_2-1+\tau_2,3e-1),
		\end{multline}
		where $\tau_i := d_i-\min\{2m_{1,i}+m_{2,i},m_{i,1}+2m_{2,i},d_i\}\geq 0, \  i=1,2.$
	\end{itemize} 	
\end{thm}

\begin{proof} The proof of Theorem \ref{thm:regstabilization} applies almost verbatim; the specific properties of this theorem are given in Proposition \ref{prop:H0BH0}, as a refinement of Theorem \ref{thm:main}. More details are given in \S \ref{subsec:residual}.
\end{proof}

Observe that the lower bounds on the degree $(\mug,\nu)$ given in Theorem \ref{thm:mainresidual} are similar to those given in Theorem \ref{thm:main} up to shifts in some partial degrees that depend on the defining degrees of the curve $\Cc$. In addition, we mention that in the case where the base locus $\Bc$ is finite (including empty), which is the case of a general map $\Psi$ of the form \eqref{eq:defPsi}, Theorem \ref{thm:regstabilization} gives a natural extension of results obtained in \cite{BBC13}. We also emphasize that our motivation to consider maps with one-dimensional base locus comes from congruences of normal lines to rational surfaces that have been introduced in Section \ref{sec:surface}.

\section{Vanishing of some local cohomology modules}\label{sec:proofmainthm}

The goal of this section is to provide results on the vanishing of particular graded components of some local cohomology modules. \blue{Indeed, as explained in Section \ref{subsec:mainres}, these vanishing results are necessary to complete the proofs of Theorem \ref{thm:regstabilization} and Theorem \ref{thm:mainresidual}}. We first recall and set some notation.  Let $k$ be a field and consider the parameterization
\begin{eqnarray}\label{eq:Psisecproof}
	\Psi: X\times\PP^1 & \dto & \PP^3 \\ \nonumber
	       \ux \times(\t:t) & \mapsto &
		   \left(\Psi_0: \Psi_1: \Psi_2:  \Psi_3  \right)(\xi;\t,t)
\end{eqnarray} 
The variety $X$ stands for $\PP^2$ or $\PP^1\times\PP^1$, so that $n=3$ or $n=4$. \blue{As in Section \ref{subsec:mainres}}, we denote by $R_X$ its coordinate ring, which is a standard graded polynomial ring, \blue{and by $B:=(\xi).(\t,t)$ its irrelevant ideal}. Thus, the polynomials $\Psi_0,\Psi_1,\Psi_2,\Psi_3$ are multi-homogeneous of multi-degree $(\dg,e)$ in the polynomial ring $R:=R_X\otimes_k R_1$ where $R_1=k[\t,t]$ is the coordinate ring of $\PP^1$. We assume that $(\dg,e)\geq (\bm{1},1)$ (otherwise the map is not dominant), where $\bm{1}:=(1,1)$ in the case $X=\PP^1\times \PP^1$. Let $I$ be the ideal generated by the coordinates of the map $\Psi$, i.e. $I:=(\Psi_0,\Psi_1,\Psi_2,\Psi_3)\subset R$. The base locus $\Bc$ of $\Psi$ is the subscheme of $X\times \PP^1$ defined by $I$. Without loss of generality, we assume that $\Bc$ is of dimension at most one.

\begin{thm}\label{thm:main} Take again the notation of \S \ref{subsec:matrixrep} and assume that one of the two following properties holds:
	\begin{itemize}
		\item[(a)] The base locus $\Bc$ is finite, possibly empty,
		\item[(b)] $\dim(\Bc)=1$, $\Cc$ has no section in degree $<(\bm{0},e)$ and $I^{\textrm{sat}}=I'^{\textrm{sat}}$ where $I'$ is an ideal generated by three general linear combinations of the polynomials $\Psi_0,\ldots,\Psi_3$.
	\end{itemize}
	Then, for any point $p$ in $Spec(R)$ such that $\Lk_p$ is finite, possibly empty, we have that 
$$H^i_B(\Sym_R(I)\otimes_R\kappa(p))_{(\mug,\nu)}=0$$	
for all integers $i$ and all degree $(\mug,\nu)$ satisfying to \eqref{eq:thmain1-P2} or \eqref{eq:thmain1-P1P1}.	
\end{thm}

In order to prove this theorem we will use several notions and results from commutative algebra. We refer the reader to the books \cite{Eis95} and \cite{bruns1998cohen} for the Koszul and Cech complexes, local cohomology modules and their properties, as well as spectral sequences. We also refer the readers to the papers \cite{HSV82,HSV83} regarding the approximation complexes that we will use in this proof to analyze the symmetric algebra of $I$. 

We begin with some preliminary results on the control of the vanishing of the local cohomology of the cycles and homology of the Koszul complex associated to the sequence of homogeneous polynomials $\Psi_0,\ldots,\Psi_3$. 

\subsection{Some preliminaries on Koszul homology} The properties we prove below hold in a more general setting than the one of Theorem \ref{thm:mainresidual}. In order to state them in this generality we introduce the following notation. 

\medskip

Let $S$ be a standard $\ZZ^r$-graded polynomial ring; it is the Cox ring of a product of projective spaces $\PP :=\PP^{n_1}\times \cdots \times \PP^{n_r}$. We suppose given a sequence $\Psi_0,\ldots,\Psi_s$ of homogeneous polynomials in $S$ and we consider their associated Koszul complex $K_\bullet :=K_\bullet (\Psi_0,\ldots,\Psi_s ;S)$. We denote by $Z_i$ and $H_i$, respectively, the cycles and homology modules of $K_\bullet$. 

The irrelevant ideal $B$ of the Cox ring $S$ is the product of the $r$ ideals defined by the $r$ sets of variables. We set $I:=(\Psi_0,\ldots,\Psi_s )\subseteq S$ and  $\Bc :=\Proj (S/I)\subseteq \PP$. Recall that, for $i\geq 2$, and any $S/I$-module $M$ with associated sheaf ${\mathcal F}$, 
$$
H^i_B (M)_\mug \simeq H^{i-1}(\Bc ,{\mathcal F}(\mug ))
$$
and in particular $H^i_B (M)=0$ for $i>\dim (\Bc )+1$. We notice that in the setting of \eqref{eq:Psisecproof} we have $\dim \Bc \leq1$, $s=3$ and either $r=2$ and $(n_1,n_2)=(2,1)$ or $r=3$ and  $(n_1,n_2,n_3)=(1,1,1)$.

\medskip

As in the theory of multi-graded regularity, it is important to  provide regions in $\ZZ^r$ where some local cohomology modules of the ring $S$, or a direct sum of copies of $S$ like $K_i$, vanish. We first give a concrete application of this idea and then we will define regions in the specific case we will be working with. We recall that the support of a graded $S$-module $M$ is defined as 
$$ \Supp (M):=\{ \mug \in \ZZ^r \ \vert\ M_\mug \not= 0\} \subset \ZZ^r.$$

\begin{prop}\label{prop:vanishHiBHj} For any integer $i$, let $\Rc_i \subseteq \ZZ^r$ be a subset satisfying  
	\begin{equation*}
\forall j\in \ZZ : \Rc_i \cap \Supp (H^{j}_B(K_{i+j}))=\emptyset.		
	\end{equation*}
Then, if $\dim \Bc \leq 1$ the following properties hold for any integer $i$:
\begin{itemize}
	\item For all $\mug\in \Rc_{i-1}$, $H^1_B(H_i)_{\mug}=0$.
	\item There exists a natural graded map $\delta_i : H^0_B (H_i)\rightarrow H^2_B (H_{i+1})$ such that $(\delta_i)_\mug $ is surjective for all $\mug\in\Rc_{i-1}$ and is injective 
	for all $\mug\in\Rc_{i}$.
	\end{itemize}	
In particular, 
$$
H^0_B(H_i)_\mug \simeq H_B^2(H_{i+1})_\mug  \textrm{ for all } \mug\in \Rc_{i-1}\cap \Rc_{i}.
$$
\end{prop}
\begin{proof} 
We consider the double complex obtained from the Koszul complex $K_\bullet$ by replacing each term by its associated Cech complex with respect to the ideal $B$: $\CcB^{\bullet}(K_\bullet)$. This double complex gives rise to two spectral sequences that both converge to the same limit. At the second step, the row-filtered spectral sequence is of the following form 
$$
\xymatrixrowsep{0.8em}
\xymatrixcolsep{1.2em}
\xymatrix{
\cdots  &H^0_B({H_{i+2}})  &H^0_B({H_{i+1}})\ar[ldd]_(.4){\delta_{i+1}} &H^0_B({H_i})  \ar[ldd]_(.4){\delta_i} &  H^0_B({H_{i-1}})  \ar[ldd]_(.4){\delta_{i-1}}&\cdots \\
\cdots  &H^1_B({H_{i+2}})&H^1_B({H_{i+1}}) &H^1_B({H_i})& H^1_B({H_{i-1}}) &\cdots   \\
\cdots  &H^2_B({H_{i+2}}) &H^2_B({H_{i+1}}) &H^2_B({H_i}) & H^2_B({H_{i-1}}) &\cdots  \\
\cdots  &0 & {0}  & 0 &  0 &  \cdots }.
$$

On the other hand, the terms of the column-filtered spectral sequence at the first stage on the diagonal whose total homology is filtered by $\ker(\delta_i)$, $\coker(\delta_{i+1})$ and $H^1_B(H_{i+1})$ are  $H^j_B(K_{i+j})$ for $j\in \NN$. It follows that :
$$
H^1_B(H_{i+1})_{\mug}=\ker(\delta_i)_{\mug}=\coker(\delta_{i+1})_{\mug}=0, \forall\mug\in \Rc_{i}.
$$
\end{proof}

\begin{rem}\label{rem:dim0Lemmavanishing} If $\dim(\Bc)=0$ then $H^2_B(H_i)=0$, for any $i$,  since $H_i$ is a $S/I$-module. Therefore, in this case Proposition \ref{prop:vanishHiBHj} shows that
$$
	H^0_B(H_i)_{\mug}=0,\ \forall \mug\in \Rc_{i}\quad \hbox{and}\quad H^1_B(H_i)_{\mug}=0,\ \forall \mug\in \Rc_{i-1}.
$$
\end{rem}

We now turn to properties on the cycles of the Koszul complex $K_\bullet (\Psi_0,\ldots,\Psi_s ;S)$.

\begin{prop}\label{prop:vanishHiBZj} Assume that $n_i>0$ for all integer $i$.  
Then, for any integer $p$ the following properties hold:
\begin{itemize}
	\item $H^0_B(Z_p)=H^1_B(Z_p)=0$,
	\item  $H^2_B(Z_p)_\mug \simeq H^0_B(H_{p-1})_\mug$ for all $\mug\in \Rc_{p-2}$,
	\item  $H^3_B(Z_p)_\mug=0$ for all $\mug\in \Rc_{p-2}\cap\Rc_{p-3}$,
	\item  $H^4_B(Z_p)_\mug = 0$ for all $\mug\in \Rc_{p-3}\cap\Rc_{p-4}$.
	\end{itemize}	
\end{prop}

\begin{proof} We consider the complex $$\mathcal{C}_{\bullet}:= 0\rightarrow Z_p \hookrightarrow  K_{p} \rightarrow K_{p-1} \rightarrow \cdots  \rightarrow K_1 \rightarrow K_0 \rightarrow 0$$ 
which is built form the Koszul complex $K_\bullet$. This complex gives rise to the double ${{\breve{\mathscr{C}}}_B}^{\bullet}(\mathcal{C}_{\bullet})$, which itself gives rise to two spectral sequences that converge to the same limit. The column-filtered spectral sequence has a  first page is of the following form:

	$$
	\xymatrixrowsep{0.6em}
	\xymatrixcolsep{1em}
	\xymatrix{
	H^0_B(Z_p) &  0  & \cdots&  0 &  0 \\
	H^1_B(Z_p) &  0  & \cdots&  0 &  0 \\
	H^2_B(Z_p) \ar[r] & H^2_B(K_p)  \ar[r]&H^2_B(K_{p-1})  \ar[r]& H^2_B(K_{p-2})   \ar[r] & \cdots \\
	H^3_B(Z_p) \ar[r] & H^3_B(K_p)  \ar[r]&H^3_B(K_{p-1}) \ar[r] & H^3_B(K_{p-2})   \ar[r] & \cdots \\
	H^4_B(Z_p) \ar[r] & H^4_B(K_p)  \ar[r]&H^4_B(K_{p-1})  \ar[r]& H^4_B(K_{p-2})   \ar[r] & \cdots \\}
	$$
On the other hand, the row-filtered spectral sequence at the second step is of the following form
	\[
	\xymatrixrowsep{0.6em}
	\xymatrixcolsep{1em}
	\xymatrix{
	{0} & {0} & H^0_B({H_{p-1}})&H^0_B({H_{p-2}}) \ar[ldd]_(.4){\delta_{p-2}}&H^0_B({H_{p-3}}) \ar[ldd]_(.4){\delta_{p-3}}& \cdots \\
	{0} & {0} & H^1_B({H_{p-1}})&H^1_B({H_{p-2}})& H^1_B({H_{p-3}})&\cdots  \\
	{0} & {0} & H^2_B({H_{p-1}}) &H^2_B({H_{p-2}})& H^2_B({H_{p-3}}) &\cdots\\
	{0} & {0} & 0& 0 &0\\
	{0} & {0} & 0& 0 &0\\
	}
	\]
	Comparing these two, and using that they have same abutment, we get the claimed results for $H^0_B(Z_p)$, $H^1_B(Z_p)$, $H^2_B(Z_p)$ and $H^4_B(Z_p)$. For $H^3_B(Z_p)$, we get that $H^3_B(Z_p)_\mug$ is filtered by  $H^1_B(H_{p-1})_\mug$ and $\ker (\delta_{p-2})_\mug$ for all $\mug\in \Rc_{p-2}\cap\Rc_{p-3}$ and the conclusion follows from Proposition \ref{prop:vanishHiBHj}.
\end{proof}

When $s=3$ and the polynomials $\Psi_0,\ldots ,\Psi_3$ are of the same degree $\bm{\delta}$ (which is equal to $(\dg,e)$ in the setting of \eqref{eq:Psisecproof}), the corresponding approximation complex $\Zc_{\bullet}$ to these polynomials is of the form 
$$
0\rightarrow \mathcal{Z}_3\rightarrow \mathcal{Z}_2\rightarrow  \mathcal{Z}_1 \rightarrow \mathcal{Z}_0\rightarrow 0,
$$ 
with $\mathcal{Z}_i =Z_i[i\bm{\delta}]\otimes S(-i)$. 
Using the Cech complex construction, we can consider the double complex $\CcB^{\bullet}(\Zc_{\bullet})$ that gives rise to two canonical spectral sequences corresponding to the row and column filtrations of this double complex. 
The graded pieces of the spectral sequence at the first step for the column filtration are $H^p_B(Z_q)_{\mug +p\bm{\delta}}$. By Proposition \ref{prop:vanishHiBZj}, and under its assumptions, if $\mug \in \Rc_{-2}$ all these modules vanish except for $(p,q)=(2,2)$ and  for $(p,q)=(2,1)$. Hence, for $\mug \in \Rc_{-2}$, $\CcB^{\bullet}(\Zc_{\bullet})_\mug$ is quasi-isomorphic to the complex
$$
0\rightarrow H^0_B({H_{1}})_{\mug +2\bm{\delta}}\otimes S(-2)\rightarrow H^0_B({H_{0}})_{\mug +\bm{\delta}}\otimes S(-1)\rightarrow 0
$$
that is in turn isomorphic to
$$
0\rightarrow H^2_B({H_{2}})_{\mug +2\bm{\delta}}\otimes S(-2)\rightarrow H^2_B({H_{1}})_{\mug +\bm{\delta}}\otimes S(-1)\rightarrow 0
$$
for $\mug \in \Rc_{-1}\subseteq \Rc_{-2}$ by Proposition \ref{prop:vanishHiBHj}, assuming in addition that $\dim(\Bc)=1$. Consequently, our next goal is to control the vanishing of the graded components of $H_B^2(H_1)$ and $H_B^2(H_2)$. \blue{We recall that the notation $\EE(-)$ has been previously introduced in Notation \ref{not:E}.}

\begin{lem}\label{lem:H2BH2} Assume that $\dim(\Bc)=1$ and that the $s+1$ forms $\Psi_0,\ldots,\Psi_s$ are of the same degree $\bm{\delta}$. Let $\Cc$ be the unmixed curve component of $\Bc$ and set  $p:=s-\dim \PP +2$ and $\bm{\sigma}:=(s+1)\bm{\delta} -(n_1 +1,\cdots ,n_r +1)$.
Then, for all $\mug\in\ZZ^r$ we have 
$$H^2_B(H_p)_\mug \simeq H^0(\Cc,\Oc_\Cc(-\mug+\bm{\sigma}  ))^\vee.$$
In particular, if $\Cc$ has no section in degree $<\mug_0$, for some $\mug_0 \in \ZZ^r$,
	then $$H^2_B(H_p)_\mug =0 \textrm{ for all } \mug \in \EE ((s+1)\bm{\delta}-(n_1 ,\cdots ,n_r )-\mug_0).$$
\end{lem}
\begin{proof}  As locally at a closed point $x\in \PP$, the $\Psi_i$'s contain a regular sequence of length $s-1$, and of length $s$ unless $x\in \Cc$, by \cite[\S 1-3]{bruns1998cohen} we have the isomorphisms 
	$$\widetilde{H_p(\bm{\sigma}) }\simeq \widetilde{\ext^{s-1}_S (S/I,\omega_S)}\simeq \widetilde{\ext^{s-1}_S (S/I_\Cc ,\omega_S)}\simeq\omega_\Cc $$
from which we deduce that
\begin{equation}\label{eq:H2BwC}
	H^2_B(H_p)\simeq \bigoplus_{\mug} H^1( \Cc  ,\omega_\Cc (\mug -\bm{\sigma})).
\end{equation}
Now, applying Serre's duality Theorem \cite[Corollary 7.7]{Hart77} we get
$$H^1(\Cc ,\omega_\Cc (\mug -\bm{\sigma}))\simeq H^0(\Cc,\Oc_\Cc(-\mug +\bm{\sigma}))^\vee,$$
which concludes the proof.
\end{proof}

\begin{lem}\label{lem:H2BH1} In the setting of Lemma \ref{lem:H2BH2},  let $s=\dim \PP$ and let $I'$ be an ideal generated by $s$ general linear combinations of the $\Psi_i$'s. If $I^{\textrm{sat}}=I'^{\textrm{sat}}$ then for all $\mug \in \ZZ^r$ there exists an exact sequence 
$$
 H^0(\Cc,\Oc_\Cc(-\mug -\bm{\delta}+\bm{\sigma}))^\vee\rightarrow H^2_B(H_{1})_\mug \rightarrow H^2_B(S/I)_{\mug -\bm{\delta}}\rightarrow 0.
 $$
In particular,  if $\Cc$ has no section in degree $<\mug_0$, for some $\mug_0 \in \ZZ^r$, we have that 
$H^2_B(H_{1})_\mug =0$ for all $\mug$ such that $\mug \in \EE (s\bm{\delta}-(n_1 ,\cdots ,n_r )-\mug_0)$ and $\mug-\bm{\delta} \in \Rc_{-2}.$
\end{lem}

\begin{proof} We will denote by $H_i'$ the $i$th homology module of the Koszul complex associated to $I'\subset R$. By \cite[Corollary 1.6.13]{bruns1998cohen} and \cite[Corollary 1.6.21]{bruns1998cohen} we have the following graded exact sequence
\begin{equation}\label{eq:exactseq}
0 \rightarrow M \rightarrow H_{1}' \rightarrow H_{1} \rightarrow H_{0}'(-\bm{\delta}) \rightarrow N \rightarrow 0
\end{equation}
with the property that the modules $M$ and $N$ are supported on $V(B)$, which implies that $H^i_B(M)=H^i_B(N)=0$ for $i\geq 1$.

Now, the column-filtered spectral sequence associated to the double complex obtained by replacing each term in the exact sequence \eqref{eq:exactseq} by its corresponding Cech complex converges to 0. The first step of this spectral sequence has three non zero lines and is of the following form 
	$$\begin{array}{ccccccccc}
	H^0_B(M) & \rightarrow & H^0_B(H_1') &\rightarrow & H^0_B(H_1) & \rightarrow & H^0_B(H_0')(-\bm{\delta}) & \rightarrow & H^0_B(N) \\
	0 & \rightarrow & H^1_B(H_1') &\rightarrow & H^1_B(H_1) & \rightarrow & H^1_B(H_0')(-\bm{\delta}) & \rightarrow & 0 \\
	0 & \rightarrow & H^2_B(H_1') &\rightarrow & H^2_B(H_1) & \rightarrow & H^2_B(H_0')(-\bm{\delta}) & \rightarrow & 0 \\
	\end{array}$$
which implies that the right part of the bottom line 
$$
 H^2_B(H_1')\rightarrow H^2_B(H_1) \rightarrow H^2_B(H_0')(-\bm{\delta}) \rightarrow  0 
 $$ is exact. 
 By Lemma \ref{lem:H2BH2}, $H^2_B(H_1')_\mug \simeq  H^0(\Cc,\Oc_\Cc(-\mug -\bm{\delta}+\bm{\sigma}))^\vee$ and,  by Proposition \ref{prop:vanishHiBHj}, $H^2_B(H_0')(-\bm{\delta})_\mug =H^2_B(S/I)(-\bm{\delta})_\mug=0$ for $\mug -\bm{\delta} \in \Rc_{-2}$. 
\end{proof}

Before closing this paragraph, we come back to the setting of Theorem \ref{thm:main} and provide explicit subsets $\Rc_i$, as these subsets are key ingredients for proving Theorem \ref{thm:main}. They can be derived from the known explicit description of the local cohomology of polynomial rings. 
More precisely, we have $s=3$ and the multi-homogeneous polynomials $\Psi_0,\ldots,\Psi_3$ are defined in the polynomial ring $S:=R=R_X\otimes R_1$ and are of degree $\bm{\delta}:=(\dg,e)\geq (\bm{1},1)$.

\begin{lem}\label{lem:locCohmoR} We have that $H^i_B(R)=0$ for all $i\neq 2,3,4$. In addition, if $X=\PP^2$ then $R_X=k[w,u,v]$ and we have that
$$H^2_B(R)\simeq R_X\otimes\check{R_1}, \ \ H^3_B(R)\simeq \check{R_X}\otimes{R_1}, \ \ H^4_B(R)\simeq \check{R_X}\otimes\check{R_1}$$
where $\check{R_1}=\frac{1}{\t t}k[\t^{-1},t^{-1}]$ and $\check{R_X}=\frac{1}{wuv}k[w^{-1},u^{-1},v^{-1}]$.

If $X=\PP^1\times\PP^1$ then $R_X=R_2\otimes R_3$, where $R_2=k[\u,u]$, $R_3=k[\v,v]$, and we have that
$$H^2_B(R)\simeq\bigoplus_{\substack{i=1\ldots 3, \\ \{i,j,k\}=\{1,2,3\}}} \check{R_i}\otimes R_j\otimes R_k,$$ 
$$H^3_B(R)\simeq\bigoplus_{\substack{i=1\ldots 3, \\ \{i,j,k\}=\{1,2,3\}}} R_i\otimes \check{R_j}\otimes \check{R_k},\ \ H^4_B(R)\simeq \check{R_1}\otimes \check{R_2}\otimes \check{R_3}$$
where $\check{R_2}$ and $\check{R_3}$ are defined similarly to $\check{R_1}$.
\end{lem}
\begin{proof} See for instance \cite[\S 6]{BOTBOL2011381}.
\end{proof}

Using the above lemma, we define the following subsets $\Rc_i$.
\begin{defn}\label{lem:vanishHiKj} 
	With notations as above, 

$\bullet$ If $X =\PP^2$ we set 
\begin{itemize}
	\item[-] $\Rc_i:=\ZZ^2 \textrm{ for all } i \notin [-4;2],$ 
	\item[-] $\Rc_{-4}:=\EE(-\infty,-1)\cup \EE(-2,-\infty),$
	\item[-] $\Rc_{-3}:=\EE(-2,e-1)\cup \EE(d-2,-\infty),$
	\item[-] for $i=-2,-1,0,$
	$$\Rc_i:=\EE ((i+3)d-2,(i+4)e-1)\cup 
		\EE ((i+4)d-2,(i+2)e-1)\subseteq \ZZ^2,$$
	\item[-] $ \Rc_1:=\EE(4d-2,3e-1),$ 
	\item[-] $\Rc_2:=\EE(-\infty,4e-1).$
\end{itemize}	

$\bullet$ If $X =\PP^1\times \PP^1$ we set
\begin{itemize}
	\item[-] $\Rc_i:=\ZZ^3 \textrm{ for all } i \notin [-4;2],$ 
	\item[-] $\Rc_{-4}:=\EE(-1,-\infty,-\infty)\cup \EE(-\infty,-1,-\infty)\cup \EE(-\infty,-\infty,-1),$
	\item[-] $\Rc_{-3}:=\EE(d_1-1,-1,-1)\cup \EE(-1,d_2-1,-1)\cup \EE(-1,-1,e-1),$
	\item[-] for $i=-2,-1,0,$
	$$\Rc_i:=\bigcup_{(a,b,c)\vert \{ a,b,c\} =\{ 2,3,4\} }\EE ((a+i)d_1-1,(b+i)d_2-1,(c+i)e-1)\subseteq \ZZ^3,$$
	\item[-] $\Rc_1:=\EE(4d_1-1,4d_2-1,3e-1)\cup\EE(4d_1-1,3d_2-1,4e-1)\cup\EE(3d_1-1,4d_2-1,4e-1),$
	\item[-] $ \Rc_2:=\EE(4d_1-1,4d_2-1,4e-1).$
\end{itemize}
\end{defn}

It is straightforward to check that these subsets satisfy to the properties required in Proposition \ref{prop:vanishHiBHj}. We notice that $\Rc_p\subseteq \Rc_{p-1}$ for all $p\leq 1$, but we also emphasize that not all these subsets are the largest possible ones in view of Lemma \ref{lem:locCohmoR}, as we have restricted ourselves to subsets that fit our needs to prove Theorem \ref{thm:main}.

\subsection{Proof of Theorem \ref{thm:main}}
We focus on the difficult case of this theorem, namely the case where the base locus $\Bc$ is not finite, which corresponds to the item (b) in its statement. If $\Bc$ is finite, then the proof simplifies as explained in Remark \ref{rem:dim0Lemmavanishing} and gives the same conclusion. In what follows, we take again the notation of \eqref{eq:Psisecproof} and Theorem \ref{thm:main}.

\medskip

We consider the approximation complex $\Zc_{\bullet}$ associated to the sequences of homogeneous polynomials $\Psi_0$, $\Psi_1$, $\Psi_2$ and $\Psi_3$. It inherits the multi-graded structure of $R$ and it has an additional grading with respect to $\PP^3=\Proj(k[x_0,x_1,x_2,x_3])$.

	 Let $\mathfrak{p}$ be a point in $\mathrm{Spec}(k[x_0,x_1,x_2,x_3])$. 
The specialization of the approximation complex $\Zc_{\bullet}$ at the point $\mathfrak{p}$ yields the complex $\Zc^\pp_\bullet:=\Zc_{\bullet}\otimes\kappa(\mathfrak{p})$ which is of the form 
$$0\rightarrow \mathcal{Z}_3\otimes\kappa(\mathfrak{p})\rightarrow \mathcal{Z}_2\otimes\kappa(\mathfrak{p})\rightarrow  \mathcal{Z}_1 \otimes\kappa(\mathfrak{p})\rightarrow \mathcal{Z}_0\otimes\kappa(\mathfrak{p})\rightarrow 0.$$ 
Notice that $H_0(\Zc_{\bullet})=\Sym_R(I)$ and $H_0(\Zc^\pp_{\bullet})=\Sym_R(I)\otimes\kappa(\pp)$.

Now, using the Cech complex construction, we can consider the double complexes $\CcB^{\bullet}(\Zc_{\bullet})$ and  $\CcB^{\bullet}(\Zc^\pp_{\bullet})$ that both give rise to two canonical spectral sequences corresponding to the row and column filtrations of these double complexes. 
As $\Cc$ (the curve component of the base locus $\Bc$) is almost a complete intersection at its generic points, $H^i_B(H_j(\Zc_{\bullet}))=0$ for $i>1$ and $j>0$. 
Let $Y\subset \Proj (R/I)$ be the locus where $I$ is not locally an almost complete intersection. Then $Y$ is either empty or of dimension zero. As $\Zc_{\bullet}$ is exact off $Y$, $\Zc^\pp_{\bullet}$ is exact off $Y\cup \pi^{-1}(\mathfrak{p})$. Therefore, it follows from our hypothesis that one also has $H^i_B(H_j(\Zc^\pp_{\bullet}))=0$ for $i>1$ and $j>0$.
From these considerations, we deduce that the row-filtered spectral sequence of the double complex $\CcB^{\bullet}(\Zc_{\bullet})$ converges at the second step and is equal to
$$	
\xymatrixrowsep{0.8em}
\xymatrixcolsep{0.8em}
\xymatrix{
{0} & H^0_B{H_3(\Zc_{\bullet})} & H^0_B(H_2(\Zc_{\bullet})) & H^0_B(H_1(\Zc_{\bullet})) & H^0_B(\Sym_R(I))\\
{0} & H^1_B{H_3(\Zc_{\bullet})} & H^1_B(H_2(\Zc_{\bullet})) & H^1_B(H_1(\Zc_{\bullet})) & H^1_B(\Sym_R(I))\\
{0} & 0 & 0 & 0 & H^2_B(\Sym_R(I))\\
{0} & 0 & 0 & 0 & 0}
$$

On the other hand, the column-filtered spectral sequence of the double complex $\CcB^{\bullet}(\Zc_{\bullet})$ at the first step is  
$$
\xymatrixrowsep{0.8em}
\xymatrixcolsep{1em}
\xymatrix{
0 & 0 & 0 & 0 & 0&0\\
{0} & 0 & 0 & 0 & 0&0\\
{0} \ar[r]& H^2_B({\Zc_3}) \ar[r]& H^2_B({\Zc_2})\ar[r]& H^2_B({\Zc_1})\ar[r]& H^2_B({\Zc_0})\ar[r]&0\\
{0} \ar[r]& H^3_B({\Zc_3}) \ar[r]& H^3_B({\Zc_2})\ar[r]& H^3_B({\Zc_1})\ar[r]& H^3_B({\Zc_0})\ar[r]&0\\
{0} \ar[r]& H^4_B({\Zc_3}) \ar[r]& H^4_B({\Zc_2})\ar[r]& H^4_B({\Zc_1})\ar[r]& H^4_B({\Zc_0})\ar[r]&0\\}
$$
It follows that $H_3(\Zc_\bullet )=0$ and $ H^0_B(H_2(\Zc_\bullet))=0$. Moreover, for all degree $(\mug,\nu) \in \Rc_{-2}$ we have $H^i_B({\Zc_j})_{(\mug,\nu)} =0$ unless $i=2$ and $j\in \{ 1,2\}$. 
Let 
\begin{equation}\label{eq:HOBHOexactseq}
\xymatrix{H^0_B({H_1})_{(\mug,\nu) +2(\dg,e)}\otimes S[-2]\ar^{\phi_{(\mug,\nu)}}[r]&  H^0_B({H_0})_{(\mug,\nu) +(\dg,e)}\otimes S[-1]\\}	
\end{equation}
be the degree $(\mug,\nu)$ component of the only potentially non zero map of this graded piece of the double complex. Then, 
\begin{itemize}
	\item $H^1_B(\Sym_R(I))_{(\mug,\nu)} =0$ if and only if $\phi_{(\mug,\nu)}$ is surjective,
	\item  $H^0_B(\Sym_R(I))_{(\mug,\nu)} =H^1_B(H_1(\Zc_{\bullet}))_{(\mug,\nu)}=0$ if and only if $\phi_{(\mug,\nu)}$ is injective.	
\end{itemize}

The same arguments apply to  $\CcB^{\bullet}(\Zc^\pp_{\bullet})$ and show that for ${(\mug,\nu)} \in \Rc_{-2}$ :
\begin{itemize}
	\item $H^1_B(\Sym_R(I)\otimes \kappa(\mathfrak{p}))_{(\mug,\nu)} =0$ if and only if $\phi_{(\mug,\nu)} \otimes \kappa(\mathfrak{p})$ is surjective,
	\item  $H^0_B(\Sym_R(I)\otimes \kappa(\mathfrak{p}))_{(\mug,\nu)} =H^1_B(H_1(\Zc^\pp_{\bullet}))_{(\mug,\nu)}=0$ if and only if $\phi_{(\mug,\nu)}\otimes \kappa(\mathfrak{p})$ is injective.	
\end{itemize}	
Furthermore, for all ${(\mug,\nu)} \in \Rc_{-1}$ the map $\phi_{(\mug,\nu)}$ identifies with the map
$$
\xymatrix{H^2_B({H_2})_{{(\mug,\nu)} +2(\dg,e)}\otimes S[-2]\ar[r]&  H^2_B({H_1})_{{(\mug,\nu)} +(\dg,e)}\otimes S[-1]\\}
$$
and
the conclusion follows from Lemma \ref{lem:H2BH2} and Lemma \ref{lem:H2BH1}. Indeed, assume first that $X=\PP^2$, then we need to have $(\mug,\nu) \in \Rc_{-1}$ and Lemma \ref{lem:H2BH2} and Lemma \ref{lem:H2BH1} both require additionally that 
$$(\mug,\nu) \in \EE(2d-2-\mu_X,2e-1-\mu_1),$$
where $\mug_0=(\mu_X,\mu_1)$ is such that the unmixed curve component $\Cc$ of the base locus has no section in degree $<\mug_0$. Observing that $\Rc_{-1}$ is precisely the expected region \eqref{eq:thmain1-P2} to prove our theorem, we must have that 
$$ \Rc_{-1} \subseteq \EE(2d-2-\mu_X,2e-1-\mu_1).$$
This latter inclusion holds if $\mu_X\geq 0$ and $\mu_1\geq e$. From here, the theorem follows as we assumed that $\Cc$ has no section in degree $<(\mu_X,\mu_1)=(0,e)$. 

Now, consider the case $X=\PP^1\times\PP^1$. Similarly, we need to have $(\mug,\nu) \in \Rc_{-1}$ and both Lemma \ref{lem:H2BH2} and Lemma \ref{lem:H2BH1} require additionally that 
$$(\mug,\nu) \in \EE(2d_1-1-{\mu_1},2d_2-1-{\mu_2},2e-1-\mu_3),$$
where $\mug_0=(\mu_1,\mu_2,\mu_3)$. If we impose, as in the case $X=\PP^2$, that $\Rc_{-1}$ is contained in the above region, then we must have $\mu_1\geq d_1$, $\mu_2\geq d_2$ and $\mu_3\geq e$. In order to have a weaker assumption on the global sections on the curve $\Cc$, we preferably choose to set $\mu_1=\mu_2=0$ and then restrict $\Rc_{-1}$ accordingly, which gives the claimed region \eqref{eq:thmain1-P1P1}. 

\subsection{Residual of a complete intersection curve}\label{subsec:residual} 
The proof of Theorem \ref{thm:main} completes the proof of Theorem \ref{thm:regstabilization}, but it still remains 
to prove Theorem \ref{thm:mainresidual}. For that purpose, we proceed as in the proof of Theorem \ref{thm:main} but with a different argument to control the vanishing of some graded components of the module $H^0_B(H_0)$ that appears in \eqref{eq:HOBHOexactseq}. 
We maintain the notation of the previous paragraph, \blue{including Notation \ref{not:E}.}

 \begin{prop}\label{prop:H0BH0} Let $J$ be an homogeneous ideal in $R$ generated by a regular sequence $(g_1,g_2)$ such that $I\subset J$ and $(I:J)$ defines a finite subscheme in $X\times \PP^1$. Denote by $(\mg_1,n_1)$, resp.~$(\mg_2,n_2)$, the degree of $g_1$, resp.~$g_2$, and define the integer $\kappa:=\min(e,n_1+n_2)$. Then,  
 $H^0_B(H_0(K_\bullet))_{(\mug,\nu)}=0$ for all $(\mug,\nu)$ satisfying to the following conditions:
 \begin{itemize}
 	\item if $X=\PP^2$, 
$$(\mu,\nu)\in \EE(4d-2-\min\{m_1+m_2,2m_1,2m_2\},3e-1-\kappa).$$
\end{itemize}

\begin{itemize}
 	\item if $X=\PP^1\times \PP^1$, 
	$$	(\mug,\nu)\in \EE(4d_1-1-\varepsilon_1,4d_2-1-\varepsilon_2',3e-1-\kappa)\cup 
		\EE(4d_1-1-\varepsilon_1',4d_2-1-\varepsilon_2,3e-1-\kappa)$$
	where for $i=1,2,$ 
	\begin{align*}
		\varepsilon_i &:=\min \{ 2m_{1,i},m_{1,i}+m_{2,i},2m_{2,i}\} , \\
		\varepsilon'_i &:= \min \{ 2m_{1,i}+m_{2,i},m_{1,i}+2m_{2,i},d_i+m_{1,i},d_i+m_{2,i}\}. 
	\end{align*}
 \end{itemize} 
 \end{prop}

\begin{proof} Since $I\subset J$, we have the canonical exact sequence
	$$ 0 \rightarrow J/I \rightarrow R/I \rightarrow R/J \rightarrow 0$$
and hence, by the associated long exact sequence of local cohomology we deduce that 
$$H^0_B(H_0(K_\bullet))_{(\mug,\nu)}=H^0_B(R/I)_{(\mug,\nu)}=0$$ 
if both 
	$H^0_B(R/J)_{(\mug,\nu)}=0$ and $H^0_B(J/I)_{(\mug,\nu)}=0$. Our objective is to analyze these two latter conditions. 
		
	Set $\bm{\sigma}:=(\mg_1+\mg_2,n_1+n_2)$. Since $J=(g_1,g_2)$ is generated by a regular sequence, its associated Koszul complex $K^J_\bullet$, 
	$$ 0 \rightarrow F_2=R(-\bm{\sigma}) \rightarrow F_1=\oplus_{i=1}^2 R(-(\mg_i,n_i)) \rightarrow F_0=R,$$	
	is acyclic. Therefore, the two classical spectral sequences associated to the double complex \blue{$\CcB^\bullet(K^J_\bullet)$} shows that $H^0_B(R/J)_{(\mug,\nu)}=0$ for all $(\mug,\nu)$ such that 
	\begin{equation}\label{eq:H0NH0-cond1}
		H^2_B(R)_{(\mug,\nu)-\bm{\sigma}}=0.
	\end{equation}
	
	Now, from the inclusion $I\subset J$ the decomposition of the $\Psi_j's$ on the $g_1,g_2$ gives the $2\times 4$-matrix $H$ such that
	$$ \left( \Psi_0 \ \Psi_1 \ \Psi_2 \ \Psi_3 \right) = \left( g_1 \ g_2 \right) 
	\left(
	\begin{array}{cccc}
	h_{0,1} & h_{1,1} & h_{2,1} & h_{3,1} \\
	h_{0,2} & h_{1,2} & h_{2,2} & h_{3,2} \\
	\end{array}
	\right)= \left( g_1 \ g_2 \right) H$$ 
	and which corresponds to the homogeneous map 
	$$K_1=R(-(\dg,e))^4 \xrightarrow{H} F_1=R(-(\mg_1,n_1))\oplus R(-(\mg_2,n_2)).$$
	From here, we obtain a finite free graded presentation of $J/I$, namely
	$$ F_1'=F_2 \oplus K_1 \xrightarrow{\varphi} F_1  
	\xrightarrow{(g_1,g_2)} J/I \rightarrow 0$$
	where the map $\varphi:F_1'\rightarrow F_1$ is defined by the matrix
	   $$\left(\begin{array}{ccccc}
	   -g_2 & h_{0,1} & h_{1,1} & h_{2,1} & h_{3,1} \\
	   g_1 & h_{0,2} & h_{1,2} & h_{2,2} & h_{3,2} \\
	   \end{array}\right).$$ 
	Consider the Buchsbaum-Rim complex $C_\bullet$ associated to $\varphi$; it is of the form (see \cite[Appendix 2.6]{Eis95} and \cite[\S 2]{CFN08} for the graded version with the appropriate shifting in degrees)
	\begin{multline*}
	 C_4= S_2(F_1^*)\otimes \wedge^5(F_1')(\bm{\sigma}) \rightarrow C_3= S_1(F_1^*)\otimes \wedge^4(F_1')(\bm{\sigma}) \rightarrow \\
C_2= S_0(F_1^*)\otimes \wedge^3(F_1')(\bm{\sigma}) \rightarrow C_1=F_1' \xrightarrow{\varphi} C_0=F_1.
	\end{multline*}
	The homology of $C_\bullet$ is supported on $\mathrm{ann}_R(J/I)=(I:_R J)$, which is assumed to define a finite subscheme in $X\times \PP^1$. Therefore, the spectral sequence corresponding to the row filtration of the double complex \blue{$\CcB^\bullet(K_\bullet)$} abuts at the second step, with the term $H^0_B(J/I)$ on the principal diagonal. Comparing with the spectral sequence corresponding to the column filtration, we deduce that  $H^0_B(J/I)_{(\mug,\nu)}=0$ for all $(\mug,\nu)$ such that 
	\begin{equation}\label{eq:H0NH0-cond2}
	H^2_B(C_2)_{(\mug,\nu)}= H^3_B(C_3)_{(\mug,\nu)}= H^4_B(C_4)_{(\mug,\nu)}=0.
	\end{equation}
	Moreover, from the definition of the free graded $R$-modules $C_i$, $i=2,3,4$, we have the graded isomorphisms
	$$ C_2 \simeq R(-3(\dg,e)+\bm{\sigma})^4 \oplus R(-2(\dg,e))^6,$$
	$$C_3 \simeq \oplus_{i=1}^2 R(-4(\dg,e)+(\mg_i,n_i)+\bm{\sigma}) \oplus_{i=1}^2 R(-3(\dg,e)+(\mg_i,n_i))^4,$$
	and 
	$$C_4 \simeq R(-4(\dg,e)+2(\mg_1,n_1))\oplus R(-4(\dg,e)+\bm{\sigma}) \oplus R(-4(\dg,e)+2(\mg_2,n_2)).$$

	From here, using the explicit computation of the local cohomology modules of the polynomial ring $R$, see Lemma \ref{lem:locCohmoR}, it is straightforward to check that the claimed regions satisfy both conditions \eqref{eq:H0NH0-cond1} and \eqref{eq:H0NH0-cond2}. Notice that it is assumed that $\bm{d}\geq \bm{m_i}\geq 0$ (component-wise if $X=\PP^1\times \PP^1$) and $e\geq 1$.
\end{proof}

From the proof of Proposition \ref{prop:H0BH0} it is clear that it is possible to list more valid regions than those that are given in the statement where we only provided the regions that are of interest for our application to the computation of orthogonal projection on rational surfaces. It turns out that listing all the possible regions can be a rather technical and cumbersome task. For instance, in the case $X=\PP^2$, we obtain the following list of eight quadrants $Q_{i.j.k}$ for which the expected property hold:
	\begin{enumerate}
		\item[(i)] $Q_{1.1.1}=\EE(4d-2-\min\{ 2m_1,m_1+m_2,2m_2\},3e-1 -\min \{ n_1+n_2,e\})$,
		\item[(ii)] $Q_{1.1.2}=\EE(4d-2-\min\{ 2m_1,m_1+m_2\},4e-1 -\min\{2n_2,n_1+n_2+e,2e\})$,
		\item[(iii)] $Q_{1.2.2}=\EE(4d-2-\min\{2m_1,m_1+2m_2,d+m_2\},4e-1 -\min\{n_1+n_2,2n_2\})$,
		\item[(iv)] $Q_{1.2.1}=\EE(4d-2-\min\{ 2m_1,2m_2\},4e-1 -(n_1+n_2))$,
		\item[(v)] $Q_{2.1.1}=\EE(4d-2-\min\{ m_1+m_2,2m_2\},4e-1 -2n_1)$,
		\item[(vi)] $Q_{2.1.2}=\EE(4d-2-(m_1+m_2),4e-1 -\min\{2n_1,2n_2\})$,
		\item[(vii)] $Q_{2.2.1}=\EE(4d-2-\min\{2m_2,,2m_1+m_2,d+m_1\},4e-1 -\min\{2n_1,n_1+n_2\})$,
		\item[(viii)] $Q_{2.2.2}=\EE(4d-2 -\min\{ m_1+2m_2,2m_1+m_2,d+m_1,d+m_2\},$\\
		         \mbox{ } \hfill $4e-1 -\min\{2n_1,n_1+n_2,2n_2\})$.
	\end{enumerate}
For the sake of completeness, we also mention that in the special case where $g_1$ has degree $(m,0)$ and $g_2$ has degree $(0,1)$, which is important for our targeted application, the union of these eight quadrants is equal to the union of two quadrants $Q_{1.1.1}$ and $Q_{1.2.1}$, which is equal to
	$$
	\EE (4d-2,3e-2)\cup \EE (4d-m-2,4e-2).
	$$	
The case $X=\PP^1\times\PP^1$ give rise to too many regions so that they can be listed exhaustively in this paper.

\medskip

\noindent {\it Proof of Theorem \ref{thm:mainresidual}.}  The proofs of Theorem \ref{thm:regstabilization}, and hence Theorem \ref{thm:main}, apply verbatim with the exception of the control of the vanishing of the module $H^0_B(H_0)_{(\mug,\nu)}$ that appears in \eqref{eq:HOBHOexactseq}. Indeed, instead of relying on Lemma \ref{lem:H2BH1} (and Proposition \ref{prop:vanishHiBHj}) to obtain regions where this module vanishes, we apply Proposition \ref{prop:H0BH0}. As we are still using Lemma \ref{lem:H2BH2}, it follows that the claimed region are obtained by intersecting the regions obtained in Theorem \ref{thm:mainresidual} with the additional constrained given by Proposition \ref{prop:H0BH0}. 
To be more precise, if $X=\PP^2$ then $(\mu,\nu)$ must satisfy to \eqref{eq:thmain1-P2} but also must satisfy the additional condition
$$ (\mu,\nu)+(d,e) \in \EE(4d-2-\min\{m_1+m_2,2m_1,2m_2\},3e-1-\kappa).$$
From here one can easily check that the claimed region satisfies these two conditions. The case $X=\PP^1\times\PP^1$ follows exactly in the same way.
\hfill $\qed$

\section{Rings of sections in a product of projective spaces}\label{sec:secRing}

In order to apply Theorem \ref{thm:regstabilization} and Theorem \ref{thm:mainresidual} in the case of a positive dimensional base locus $\Bc$, it is necessary to analyze when the curve component of $\Bc$ has no section in bounded above degrees (see Definition \ref{def:negsec}). In this section, we provide classes of curves that satisfy to this property. Hereafter, a curve is understood to be a scheme purely of dimension one.

\medskip

To begin with, if $Z$ is a product (over the field $k$) of two schemes we  recall the following classical property which is a consequence of the K\"unneth formula.

\begin{lem}\label{lem:secProd} 
	Assume $Z$ is the product of two schemes $Z_1\subseteq \PP_1$ and $Z_2\subseteq \PP_2$, where $\PP_1$ and $\PP_2$ are two products of projective spaces. Then, setting $\PP :=\PP_1\times \PP_2$, for any degree $\mug$ and any degree $\nug$ we have
	$$
	H^0(Z,\mathcal{O}_Z (\mug ,\nug))=H^0(Z_1,\mathcal{O}_{Z_1} (\mug ))\otimes_k H^0(Z_2,\mathcal{O}_{Z_2} (\nug)).
	$$
\end{lem}

From this property, it is interesting to identify classes of subschemes in a single projective space that have no section in negative degrees, because then this provides new classes of subschemes in product of projective spaces with no section in negative degrees by product extensions. 

\medskip

 If $Z$ is a reduced irreducible subscheme in a projective space $\PP^n$ (over a field) of positive dimension, then it is well known that  $\Oc_{Z}(\mu)$ has no non-zero global sections for any integer $\mu<0$. We show that this result extends to a subscheme of a product of projective spaces.

Let $Z$ be a subscheme in a product of projective spaces $\PP:=\PP^{n_1}\times\cdots\times \PP^{n_r}$, $n_i\geq 1$ for all $i$. 
Let $R$ be the standard multi-graded ring defining $\PP$ and let 
$I\subset R$ be the multi-homogeneous defining ideal of $Z$. The ring of sections of $Z$ sits in the exact sequence 
$$
\xymatrix{
	0\ar[r]&H^0_B(R/I)\ar[r]&R/I\ar[r]&\bigoplus_{\mug \in \ZZ^r}H^0(Z,\mathcal{O}_Z (\mug ))\ar[r]&H^1_B(R/I)\ar[r]&0\\
}
$$
where $B$ stands for the ideal generated by $R_{(1,\ldots ,1)}$. 
It shows in particular that $Z$ determines $I_Z=I+H^0_B(R/I)$, the unique ideal saturated with respect to $B$ that defines $Z$, giving a one to one correspondence between $B$-saturated multi-graded ideals and subschemes of $\PP$. 

\begin{prop}\label{prop:vanishred-negsec} 
	With the above notation, assume that $Z$ is reduced with only components of positive dimension, then $H^0(Z,\mathcal{O}_Z (\mug ))=0$ for all $\mug <\bm{0}$. 
\end{prop}
\begin{proof}
First, there is a canonical inclusion 
$$H^0(Z,\mathcal{O}_Z (\mug ))\hookrightarrow \bigoplus_{i=1}^t \bigoplus_{\mug \in \ZZ^r}H^0(Z_i,\mathcal{O}_{Z_i} (\mug ))$$
where $Z_i$, for $i=1,\ldots ,t$, are the irreducible components of $Z$. Hence we can, and will, assume that $Z$ is reduced and irreducible.

Notice then that the multi-graded ring  $A=\oplus_{\mug \in \ZZ^r}H^0(Z,\mathcal{O}_Z (\mu ))$ is a domain, as it sits in the fraction field of $A$. Let $d:=\dim Z$. By Serre duality, for any $\mug$ we have  
$$H^0(Z,\mathcal{O}_Z (\mug ))\subseteq H^0(Z,\mathcal{O}_{\tilde Z} (\mug ))\simeq H^d(Z,\omega_Z (-\mug )),$$
where $\tilde Z$ is the $S_2$-ification of $Z$, namely the scheme defined by $\tilde A= {\rm End}(\omega_A)$, which is itself sitting in the integral closure of $A$. In fact $\tilde A=\oplus_\mug H^0(Z,\mathcal{O}_{\tilde Z} (\mug ))$ as $\tilde A$ satisfies $S_2$.

Now, if $H^0(Z,\mathcal{O}_Z (\mug ))\not= 0$ for some $\mug <0$, let $0\not= a\in A_\mug$. As $A$ is a domain, $a^k\not= 0$ for all $k\geq 1$. This implies that $H^d(Z,\omega_Z (-k\mug ))\not= 0$ for any $k$, contradicting Serre vanishing theorem (as $-\mu_1H_1...-\mu_r H_r$ is very ample since $\mu_i<0$ for all $i$).
\end{proof}

\begin{rem} For the sake of completeness, we mention that 
	there exist other classes of projective schemes $Z\subseteq \PP^n$ of positive dimension that have no global section in negative degrees. This is for instance the case in the following situations 
\begin{itemize}
\item  $R/I_Z$ is has depth at least 2, which in turn holds if $I_Z$ is a complete intersection or the power of such an ideal. This condition is also satisfied for large classes of ideal, in particular for determinantal ideals of many types.
\item  $Z$ is scheme defined by an intersection of symbolic powers of prime ideals (with no inclusion between any pair of these). This includes powers of locally complete intersection reduced schemes and so called fat structures on reduced equidimensional schemes.
\end{itemize}
\end{rem}

\section{Computing orthogonal projection of points onto a rational surface }\label{sec:Computation}

In this section, based on our previous results we derive a new method for computing the orthogonal projections of a point $p\in \RR^3$ onto a rational surface $\Sc\in\RR^3$. Suppose that $\Sc$ is parameterized by a map $\Phi:X \dasharrow \PP^3$, as defined in Section \ref{subsec:homphi}, with polynomials of degree $\dg$  and 
consider its associated congruence of normal lines $\Psi: X\times \PP^1 \dashrightarrow \PP^3$, as described in Section \ref{subsec:homparam}. We will compute the orthogonal projections $q_i \in\Sc$, $i=1,\ldots,r_p$, of $p$ by means of the matrices $\MM_{(\mug,\nu)}(\Psi)$, defined in Section \ref{subsec:matrixrep}. The degree of the defining polynomials of $\Psi$  will be denoted by $(\bm{\delta},1)$; their values in terms of the type of $\Phi$ are given in  Table \ref{tab:degreePsi}. 

In what follows, we assume that $d\geq 2$ if $X=\PP^2$ and that $d_1\geq 1$ and $d_2\geq 1$ if $X=\PP^1\times \PP^1$.

\subsection{Matrix representations of linear fibers}\label{subsec:MrepOfLinearFibers}
Consider the family of matrices $\MM_{(\mug,\nu)}$ associated to $\Psi$ that we introduced in Section \ref{subsec:matrixrep}. In Section \ref{sec:mrepsurface} it is proved under suitable assumptions that the corank of $\MM_{(\mug,\nu)}(p)$, $p$ a point in $\PP^3$, gives a computational representation of the linear fiber $\Lk_p$ of $p$, providing that $\Lk_p$ is finite and $(\mug,\nu)$ satisfies to some conditions. Thus, the matrix $\MM_{(\mug,\nu)}$ provides a universal matrix-based  representation of the finite linear fibers of $\Psi$.

\subsubsection{Admissible degrees}
Recall that $\Phi$ is defined by polynomials of degree $\bm{d}$ over $X$ and that its congruence of normal lines $\Psi$ is defined by polynomials of degree $(\bm{\delta},1)$ over $X\times \PP^1$. \blue{We also recall that the notation $\EE(-)$ has been previously introduced in Notation \ref{not:E}.}

\begin{cor}\label{cor:adminssibledegree} For a general parameterization $\Phi$ and a degree $(\mug,\nu)$, the corresponding matrix $\MM_{(\bm{\mu},\nu)}(\Psi)$ yields a matrix representation of the finite linear fibers of $\Psi$ providing
	\begin{align}\label{eq:mu0}
	  \bullet \ &  (\mu,\nu)\in \EE(3\delta-2,0) \textrm{ if } X=\PP^2,  \\ \nonumber
	   \bullet \ & (\mu_1,\mu_2,\nu)\in \EE(3\delta_1-1,2\delta_2-1+d_2,0) \textrm{ or } \\ \nonumber
	   & (\mu_1,\mu_2,\nu)\in \EE(2\delta_1-1+d_1, 3\delta_2-1,0) \textrm{ if } X=\PP^1\times\PP^1.  
	\end{align}
\end{cor}

\begin{proof} 
The base locus $\Bc$ of $\Psi$ is of positive dimension and we denote by $\Cc$ its unmixed curve component. Since $\Phi$ is a general parameterization, Lemma \ref{lem:baselocus} shows that $\Cc$ is a complete intersection curve defined by an ideal $J=(g_1,g_2)$ such that $\deg(g_1)=(\bm{\delta}-\bm{d},0)$ and $\deg(g_2)=(\bm{0},1)$. Then, from the results given in Section \ref{sec:secRing} we deduce that $\Cc$ has no section in degree $< (\bm{0},1)$. Therefore, Theorem \ref{thm:mainresidual} applies under our assumption. To recover the claimed bounds for the integers $(\bm{\mu},\nu)$, we observe that in our setting we have $e-n_1-n_2=1-0-1=0$, which proves the case $X=\PP^2$. If $X=\PP^1\times \PP^1$, then we have $\tau_i=\delta_i-(\delta_i-d_i)=d_i$, which concludes the proof. 
\end{proof}

Notice that we actually proved that the above corollary holds for any parameterization $\Phi$ such that its corresponding congruence of normal lines $\Psi$ satisfies the assumptions of Theorem \ref{thm:regstabilization} or Theorem \ref{thm:mainresidual}. 

\medskip

From a computational point of view, the fact that $\nu$ can be chosen to be equal to zero is extremely important and justifies the theoretical developments in Section \ref{sec:proofmainthm}. Indeed, since $\nu=0$ the cokernel of the corresponding matrix is defined solely on $X$, and not on $X\times \PP^1$. As a consequence, the orthogonal projections can be computed directly on the surface without computing their positions on normal lines (see \S \ref{subsec:comporthproj}). In Table \ref{tab:admissibledegree} we give the precise value of the lowest admissible degree, denoted $\bm{\mu}_0$, in terms of the type of surface parameterized by $\Phi$.

\begin{table}[ht]
\centering	
\renewcommand{\arraystretch}{1.5}
\begin{tabular}{|c|c|c|}  
	\hline 
	$\bm{\mu}_0$ &  Triangular surface &  Tensor-product surface \\ \hline
	Non-rational & $6d-8$ & $(6d_1-4,5d_2-3)$ or $(5d_1-3,6d_2-4)$ \\ \hline 
	Rational & $9d-11$ & $(9d_1-7,7d_2-5)$ or $(7d_1-5,9d_2-7)$ \\  \hline
\end{tabular}
\caption{Lowest admissible degrees $\bm{\mu}_0$ for building matrix representations depending on the type of surfaces, namely non-rational/rational triangular/tensor-product surfaces.}
\label{tab:admissibledegree}
\end{table}
The above theoretical results lead us to use in practice the matrix $\MM_{(\mug_0,0)}$ as a matrix representation of the congruence of normal lines. In what follows we will denote this matrix by $\MM$ for simplicity. We recall that its entries are linear forms in $k[x_0,\ldots,x_3]$, so that 
$\MM=\sum_{i=0}^3 x_i\MM_i$ where $\MM_i$ are matrices with entries in $k$.

\subsubsection{Computational aspects} The computation of $\MM$, equivalently of the matrices $\MM_0,\ldots,\MM_3$, amounts to solve the linear system formed by the syzygies of $\Psi_0,\ldots,\Psi_3$ of degree $(\bm{\mu}_0,0)$ (see Table \ref{tab:admissibledegree}). If the parameterization $\Phi$ is given with exact coefficients, i.e.~if $k=\QQ$, then the computations can be performed over $\QQ$, otherwise they are done with floating point numbers (using the MPFR library, with a precision of 53 bits) which means that there is no certification on the results as we are dealing with numerical approximations in the computations; in what follows we summarize this setting but writing $k=\RR$. \blue{Nevertheless, we emphasize that in the case $\Phi$ is given with exact coefficients, the computation of $\MM$ with floating numbers, using the powerful singular value decomposition, provides an accurate and robust numerical approximation of its exact counterpart; we refer the interested reader to \cite[\S 5.2]{Bus14} for precise statements and more details.}

For a general rational cubic surface, $\MM$ can be computed in about 28s over $\QQ$ and about 0.9s over $\RR$. We notice that this gap in the computation time is expected because of the growth of the heights of matrix entries over $\QQ$ (see for instance \cite[\S 1]{krick2001} for more details on this topic). For a general rational bi-quadratic surface, $\MM$ is computed in about 31s over $\QQ$ whereas it is computed in about 0.16s over $\RR$. More results are given in Table \ref{tab:MrepTriSurf} and Table \ref{tab:MrepTensorProSurf}. These computations have been made with the software {\sc Macaulay2} for the computations over $\QQ$, and the software {\sc SageMath} for the computations over $\RR$.
\begin{table}[ht]
	\centering	
	\renewcommand{\arraystretch}{1}
\begin{tabular}{|c|crr|crr|}
	\hline
 & \multicolumn{3}{c|}{non-rational} & \multicolumn{3}{c|}{rational} \\
\cline{2-7} 
   & \small{matrix} &\small{time (ms)} &  \small{time (ms)} & \small{matrix}  &\small{time (ms)} & \small{time (ms)}   \\      
  \small{$\deg(\Phi)$}&\small{size} & \small{over $\RR$} &  \small{over $\QQ$}  & \small{size} &\small{over $\RR$} & \small{over $\QQ$}  \\ 
  \hline \hline 
 $2$ & $15\times 7$  & 3 & 19 &  $36\times 29$ & 15 & 267   \\
 $3$ & $66\times 51$ & 42 & 887  &  $153 \times 150$ & 301 & 28090  \\
 $4$ & $153\times 132$ & 315 & 32473 & $351 \times 363$ & 2952 &   \\
\hline \hline
\end{tabular}
	\caption{Size and computation time in milliseconds (rounded to the millisecond) of $\MM$ over $\RR$ and $\QQ$ of non-rational and rational triangular surfaces.}
	\label{tab:MrepTriSurf}
\end{table}
\begin{table}[ht]
	\centering	
	\renewcommand{\arraystretch}{1}
\begin{tabular}{|c|crr|crr|}
			\hline
		& \multicolumn{3}{c|}{non-rational} & \multicolumn{3}{c|}{rational} \\
		\cline{2-7} 
		& \small{matrix} &\small{time (ms)} &  \small{time (ms)} & \small{matrix}  &\small{time (ms)} & \small{time (ms)}   \\      
		\small{$\deg(\Phi)$}&\small{size} & \small{over $\RR$} &  \small{over $\QQ$}  & \small{size} &\small{over $\RR$} & \small{over $\QQ$}  \\ 
		\hline \hline 
		$(1,1)$ & $9\times 5$  & 1 & 6    & $9\times 4$ & 1 & 8  \\
		$(1,2)$ & $24\times 16$  & 4 &  32  &  $30\times 20$ & 7 & 125 \\
		$(1,3)$ & $39\times 27$ & 12 &  136  & $51 \times 36$ & 21 &  1082 \\
		$(2,2)$ & $72\times 59$ & 44 &  1460  & $120 \times 108$ & 157&  31182 \\
		$(2,3)$ & $117\times 98$ & 142 &  10867  &  $204 \times 188$& 663 &   \\
		$(3,3)$ & $195\times 169$ & 575 & 96704  & $357 \times 340$ & 3353 &  \\
		\hline \hline
	\end{tabular}
	\caption{Size and computation time in milliseconds (rounded to the millisecond) of $\MM$ over $\RR$ and $\QQ$ of non-rational and rational tensor\blue{-}product surfaces in given degrees.}
	\label{tab:MrepTensorProSurf}
\end{table}

Finally, we emphasize that the corank of $\MM(p)$ for a general point $p$ and a general surface parameterization $\Phi$, in which case the linear fiber $\Lk_p$ and the fiber $\Fk_p$ coincide, is equal to the Euclidean distance degree of $\Phi$, as already mentioned in \S \ref{subsec:homphi}. Thus, our method provides a numerical approach for computing the Euclidean distance degree of the algebraic rational surface parameterized by $\Phi$ (see Table \ref{tab:TimingInversionTri} and Table \ref{tab:TimingInversionTP}).

\subsection{Computation of the orthogonal projections}\label{subsec:comporthproj} 
Given a surface parameterization $\Phi$, the matrix $\MM$ is computed only once and stored. It provides a universal representation of the finite linear fibers of the corresponding congruence of normal lines $\Psi$. More precisely, given a point $p=(p_0:p_1:p_2:p_3)\in \RR^3$, the cokernel of $\MM(p)=\sum_{i=0}^3 p_i\MM_i$ gives a linear representation of the linear fiber $\Lk_p$. Thus, if this fiber is finite then classical methods allow to compute the pre-images of $p$ via $\Psi$ by means of numerical linear algebra techniques such as singular value decompositions and eigen-computations (see for instance
 \cite{CLO98,Cox05,BKM05,DREESEN20121203}, or \cite{MC,Bus14,Shen:2016:LNS:3045889.3064444} in a setting which is similar to ours). In what follows, we describe the main steps of an algorithm for computing these orthogonal projections. It is essentially an adaptation to our context of the inversion algorithm \cite[\S 2.4]{Shen:2016:LNS:3045889.3064444} where similar matrices are used for the computation of the intersection points between 3D lines and trimmed NURBS surfaces.  

\vspace{1em}
\hrule
\vspace{.5em}
\noindent{\bf Input:} A point $p\in \RR^3$, a matrix representation $\MM$ of $\Psi$ (over $\RR$) and a numerical tolerance $\epsilon$ (default value is equal to $10^{-8}$).

\medskip

\noindent{\bf Output:} The affine parameters $(u_i,v_i)$ of $\phi$, $i=1,\ldots,l$, of the points $q_i$ that are the orthogonal projections of $p$ onto $\Sc$.
\vspace{.5em}
\hrule
\vspace{.5em}

\noindent{\bf Step 1:} Evaluate $\MM$ at the point $p\in\RR^3$ and denote by $r$, respectively $c$, its number of rows, respectively columns. 

\medskip

\noindent{\bf Step 2:} Compute the approximate cokernel of $\MM(p)$ via a singular value decomposition $\MM(p)=USV^T$: inspecting the singular values, the numerical rank  $s$ of $\MM(p)$ is obtained within the tolerance $\epsilon$. Then, the submatrix $K$ of $U$ corresponding to the last $k:=r-s$ rows is a $k\times r$-matrix whose rows form a basis of an approximate cokernel of $\MM(p)$. By construction, the columns of $K$ are indexed by the polynomial basis used in the computation the $\MM$. For simplicity, we assume that this basis, denoted $B$, is the following one:
\begin{align*}
& \textrm{if } X=\PP^2,  B:=\{ 1,u,u^2,\ldots,u^{\mu},v,vu,vu^2,\ldots,vu^{\mu},\ldots,v^{\mu} \}, \\
& \textrm{if } X=\PP^1\times\PP^1, B:=\{1,u,u^2,\ldots,u^{\mu_1},v,vu\ldots,vu^{\mu_1}\ldots,v^{\mu_2},v^{\mu^2}u,\ldots,v^{\mu_2} u^{\mu_1}\},
\end{align*}
where $\mu$ or $(\mu_1,\mu_2)$ are the chosen degrees to build $\MM$. 

\medskip

\noindent{\bf Step 3:} The matrix $K^T$ being of rank $k$, we extract a full-rank $k\times k$ - submatrix $M_1$ from $K^T$ (for instance by means of a LU-decomposition). Its rows are indexed by an ordered subset $B\textprime$ of $B$. Then, we choose another submatrix $M_2$ of $K^T$ corresponding to the rows indexed by the ordered set $v\times B\textprime$ (multiplication is member-wise). In case $v.B\textprime$ is not contained in $B$ then $\MM$ has to be rebuilt by increasing by one the degree with respect to $u$ used to build $\MM$ and we go back to Step 1. 

\medskip

\noindent{\bf Step 4:} Compute the generalized eigenvalues and eigenvectors of the pencil of matrices $M_1-vM_2$; among these generalized eigenvalues the $v$-coordinates of the orthogonal projections of $p$ are found. Then we filter them to keep those eigenvalues $v_1,\ldots,v_l$ that are real numbers and that are contained in the parameter domain of interest, typically $[0;1]$ (within a given tolerance). 

\medskip

\noindent{\bf Step 5:} For each generalized eigenvalue $v_i$, $i=1,\ldots,l$, extract from its corresponding generalized eigenvector the $u$-coordinate $u_i$ of a pre-image point of $p$ under $\Psi$, which is done by  computing the ratio of the two first coordinates of this eigenvector (so that the corresponding ratio of monomials in $B$ is equal to $u$). Finally, if $u_i \in [0;1]$, check\footnote{This last check is necessary because of the existence of ``ghost points" in the cases where the linear fiber is not the fiber.} 
if the point $\phi(u_i,v_i)$ in an orthogonal projection of $p$ on $\Sc$ within the tolerance $\epsilon$. If this is the case, then return this point as an orthogonal projection of $p$ on the surface $\Sc$. 

{
\begin{exmp}\label{ex:segre} We take again Example \ref{ex:segreMrep} and compute the orthogonal projections of the point  $p=(0,0,2)\in \RR^3$ onto \red{the Segre surface in $\PP^3$ paramterized by $\Phi$}. At the point $p$, we have $\corank(\MM_{(2,2)}) (p)=5$, which is equal to \red{the EDdegree of a non-rational bilinear tensor-product surface (see Table \ref{tab:EDdeg})}. The computation of the cokernel of $\MM_{(2,2)}(p)$ returns the following matrix

{\small
	  $$K:=\begin{pmatrix}
      2&0&0&0&1&0&0&0&0\\
      0&2&0&{-1}&0&1&0&0&0\\
      {-1}&0&1&0&0&0&1&0&0\\
      0&{-1}&0&2&0&0&0&1&0\\
      0&0&0&0&0&0&0&0&1\end{pmatrix}$$
}

\noindent whose columns are indexed by the monomial basis $B:=\{u^2v^2,u^2v,u^2,uv^2,uv,u,v^2,v,1\}.$ 
According to the shape of $B$, we define $M_1$ as the {first} 5 rows of $K^T$ and $M_2$ as the last 5 rows, except the very last one, of $K^T$. Then solving for the eigenvectors of the pencil $(M_2,M_1)$ with {\sc Macaulay2} (via the command \verb|eigenvectors(M2)| since $M_1$ is the identity matrix in this case) we get the following list of eigenvalues 
$$\{ 0, {1.73205}\,\mathbf{i}, {-{1.73205}\,\mathbf{i}}, -{1}, {1} \}$$
and their corresponding eigenvectors sorted by columns
{\small
$$\begin{pmatrix}
      {0}&{.6}&{.6}&{.447214}&{.447214}\\
      {0}&-{2.04112}\cdot 10^{-16}+{.34641}\,\mathbf{i}&-{2.04112}\cdot 10^{-16}-{.34641}\,\mathbf{i}&{-{.447214}}&{.447214}\\
      {0}&-{.6}-{5.98286}\cdot 10^{-16}\,\mathbf{i}&-{.6}+{5.98286}\cdot 10^{-16}\,\mathbf{i}&{.447214}&{.447214}\\
      {0}&{8.63551}\cdot 10^{-17}-{.34641}\,\mathbf{i}&{8.63551}\cdot 10^{-17}+{.34641}\,\mathbf{i}&{-{.447214}}&{.447214}\\
      1&{.2}+{1.17844}\cdot 10^{-16}\,\mathbf{i}&{.2}-{1.17844}\cdot 10^{-16}\,\mathbf{i}&{.447214}&{.447214}\end{pmatrix}.$$
}

\noindent From here, we see that there are three real orthogonal points whose $u$-parameter values are $u_1=0.0, u_2=-1.0$ and $u_3=1.0$. Then, using the two last rows of the eigenvectors matrix, we deduce their $v$-parameter values: $v_1=0/1=0.0$, $v_2=-.447214/.447214=-1.0$ and $v_3=.447214/.447214=1.0$. The real orthogonal projections of $p$ on the Segre surface are then obtained as $\Phi(u_i,v_i)$, $i=1,2,3$. These computations are illustrated with Figure \ref{fig:segrePicture}.
\begin{figure}[!h]
    \centering
	\includegraphics[width=.25\linewidth]{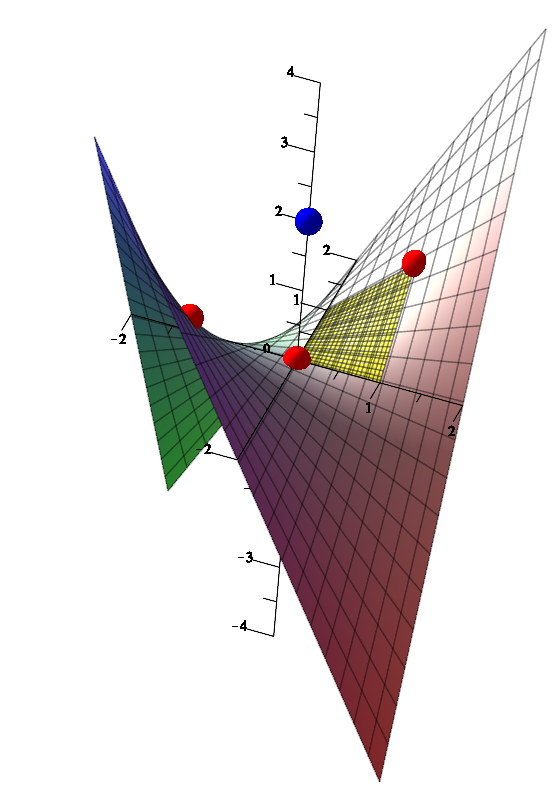}
	\caption{Orthogonal projections (red color) of the affine point $(0,0,2)$ (blue color) on the Segre variety. The surface patch corresponding to the parameter domain $(u,v)\in [0;1]^2$ is highlighted in yellow.}
	\label{fig:segrePicture}
\end{figure}
\end{exmp}
}

\subsection{Experiments}
The algorithm described in Section \ref{subsec:comporthproj} has been implemented in the software {\sc SageMath}. In Table \ref{tab:TimingInversionTP}, respectively Table \ref{tab:TimingInversionTri}, we report on the computation time to inverse a general point of a general non-rational and rational tensor\blue{-}product, respectively triangular surfaces. All the computations are done approximately and it is assumed that the matrix representation $\MM$ has already been computed and stored. 
\begin{table}[ht]
	\centering	
	\renewcommand{\arraystretch}{1}
\begin{tabular}{|c|crr|crr|}
	\hline
	& \multicolumn{3}{c|}{non-rational} & \multicolumn{3}{c|}{rational} \\
	\cline{2-7} 
	$\deg(\Phi)$ & \small{size} &\small{EDdeg} &  \small{time (ms)} & \small{size}  &\small{EDdeg} & \small{time (ms)}   \\      
	\hline \hline 
		$2$ & $15\times 7$ & 9 & 3 &  $36\times 29$ &13 & 4  \\
		$3$ & $66\times 51$ & 25 & 5 &  $153 \times 150$ & 39 & 13   \\
		$4$ & $153\times 132$ & 49 & 18 & $351 \times 363$ &79 & 113  \\
		\hline \hline
	\end{tabular}
	\caption{Size and computation time in \blue{milliseconds (rounded to ms)} of the inversion of $\MM(p)$ of non-rational and rational triangular surfaces.}
	\label{tab:TimingInversionTri}
\end{table}
\begin{table}[ht]
	\centering	
	\renewcommand{\arraystretch}{1}
\begin{tabular}{|c|crr|crr|}
	\hline
	& \multicolumn{3}{c|}{non-rational} & \multicolumn{3}{c|}{rational} \\
	\cline{2-7} 
	$\deg(\Phi)$ & \small{size} &\small{EDdeg} &  \small{time (ms)} & \small{size}  &\small{EDdeg} & \small{time (ms)}   \\      
	\hline \hline 
 $(1,1)$ & $9\times 5$  & 5 & 2& $9\times 4$  &6 & 1 \\
 $(1,2)$ & $24\times 16$  &11 & 2 &  $30\times 20$  &14 & 2 \\
 $(1,3)$ & $39\times 27$ &17 & 3 & $51 \times 36$ & 22 & 3  \\
 $(2,2)$ & $72\times 59$ &25 & 4 & $120 \times 108$ & 36 & 10 \\
 $(2,3)$ & $117\times 98$ &39 & 14&   $204 \times 188$ & 58 & 29  \\
 $(3,3)$ & $195\times 169$ &61 & 76& $357 \times 340$ &  94 &  137  \\
\hline \hline
\end{tabular}
	\caption{Size and computation time in \blue{millisecond (rounded to ms)} of the inversion of $\MM(p)$ of non-rational and rational tensor{\color{red}-}product surfaces.}
	\label{tab:TimingInversionTP}
\end{table}

The method we introduced is particularly well adapted to problems where intensive orthogonal projection computations have to be performed on the same geometric model \blue{(e.g.~surface intersections, medial axis constructions, etc; we refer to \cite[Chapter 7]{PaMa02} and the references therein)}, because one can take a great advantage of the pre-computation of the matrix representation $\MM$. Indeed this matrix allows to rely on powerful and robust numerical tools of linear algebra; see for instance Figure \ref{fig:singSurface} where orthogonal projections are close to the self-intersection locus of the surface. 

Finally, we mention that the method proposed by Thomassen et al.~\cite{TJD05} is also based on the use of congruences of normal lines, but on the algebraic side they use high degree equations of the Rees algebra associated to the defining polynomials of the congruence map (called \emph{moving surfaces}) (see \cite[\S 3]{TJD05}), which make the computations heavy in terms of time and memory. In our approach, \blue{We overcome this difficulty using the results} in Section \ref{sec:proofmainthm} that allows us to use low degree syzygies, i.e.~equations of the above Rees algebra that are \emph{linear} in the space variables (i.e.~\emph{moving planes}, and not moving surfaces of high degree in the space variables).

\begin{figure}[!htb]
        \centering
	     \includegraphics[width=.3\linewidth]{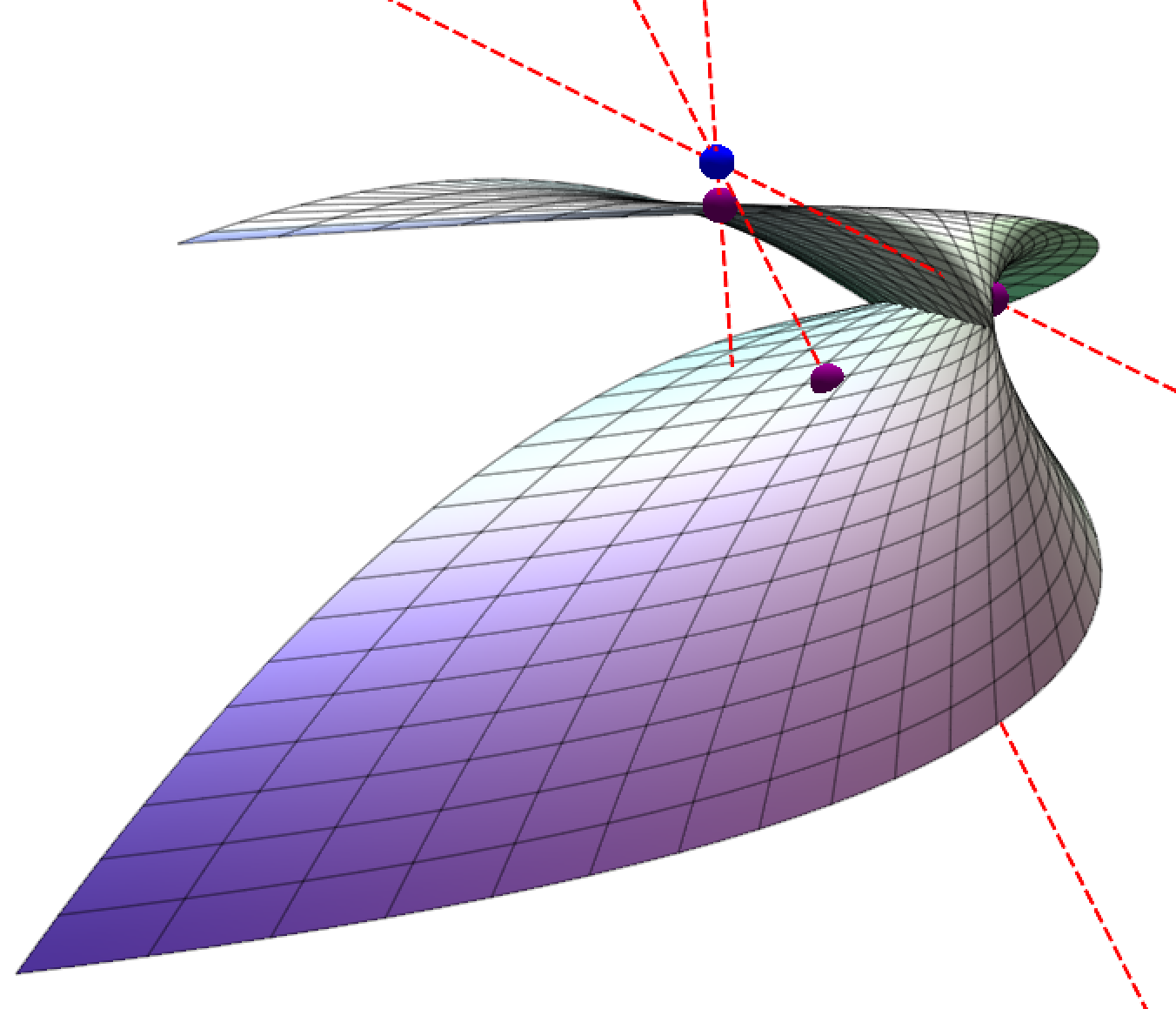}
	     \caption{Orthogonal projections of a point onto a non-rational bi-quadratic surface patch close to its self-intersection locus.}
	   \label{fig:singSurface}
\end{figure}

\section*{Acknowledgement}
This project has received funding from the European Union’s Horizon $2020$ research and innovation programme under the Marie Sk\l odowska-Curie grant agreement No $675789$. Thanks also go to the authors and contributors of the software {\sc Macaulay2} \cite{M2} that has been very helpful to compute many \red{enlightening} examples in this work.


\begin{thebibliography}{10}

\bibitem{BOTBOL2011381}
Nicol\'{a}s Botbol.
\newblock The implicit equation of a multigraded hypersurface.
\newblock {\em Journal of Algebra}, 348(1):381 -- 401, 2011.

\bibitem{BBC13}
Nicol{\'a}s Botbol, Laurent Bus{\'e}, and Marc Chardin.
\newblock Fitting ideals and multiple points of surface parameterizations.
\newblock {\em Journal of Algebra}, 420:486 -- 508, 2014.

\bibitem{BotCh}
Nicol\'{a}s Botbol and Marc Chardin.
\newblock Castelnuovo {M}umford regularity with respect to multigraded ideals.
\newblock {\em J. Algebra}, 474:361--392, 2017.

\bibitem{BD15}
Nicol{\'a}s Botbol and Alicia Dickenstein.
\newblock Implicitization of rational hypersurfaces via linear syzygies: A
  practical overview.
\newblock {\em Journal of Symbolic Computation}, 26(74):493---512, 2016.

\bibitem{BDD}
Nicol{\'a}s Botbol, Alicia Dickenstein, and Marc Dohm.
\newblock Matrix representations for toric parametrizations.
\newblock {\em Comput. Aided Geom. Design}, 26(7):757--771, 2009.

\bibitem{bruns1998cohen}
Winfried Bruns and H.~J{\"u}rgen Herzog.
\newblock {\em Cohen-Macaulay Rings}.
\newblock Cambridge Studies in Advanced Mathematics. Cambridge University
  Press, 1998.

\bibitem{Bus14}
Laurent Bus{\'e}.
\newblock Implicit matrix representations of rational b{\'e}zier curves and
  surfaces.
\newblock {\em Computer-Aided Design}, 46:14 -- 24, 2014.
\newblock 2013 SIAM Conference on Geometric and Physical Modeling.

\bibitem{BCD03}
Laurent Bus{\'e}, David Cox, and Carlos D'Andrea.
\newblock Implicitization of surfaces in {${\Bbb P}\sp 3$} in the presence of
  base points.
\newblock {\em J. Algebra Appl.}, 2(2):189--214, 2003.

\bibitem{BDissac}
Laurent Bus{\'e} and Marc Dohm.
\newblock Implicitization of bihomogeneous parametrizations of algebraic
  surfaces via linear syzygies.
\newblock In {\em I{SSAC} 2007}, pages 69--76. ACM, New York, 2007.

\bibitem{BJ}
Laurent Bus{\'e} and Jean-Pierre Jouanolou.
\newblock On the closed image of a rational map and the implicitization
  problem.
\newblock {\em J. Algebra}, 265(1):312--357, 2003.

\bibitem{BKM05}
Laurent Bus{\'e}, Houssam Khalil, and Bernard Mourrain.
\newblock Resultant-based methods for plane curves intersection problems.
\newblock In {\em Proceedings of the 8th International Conference on Computer
  Algebra in Scientific Computing}, CASC'05, pages 75--92, Berlin, Heidelberg,
  2005. Springer-Verlag.

\bibitem{BL10}
Laurent Bus{\'e} and Thang Luu~Ba.
\newblock {Matrix-based Implicit Representations of Rational Algebraic Curves
  and Applications}.
\newblock {\em Computer Aided Geometric Design}, 27(9):681--699, 2010.

\bibitem{Cha06}
Marc Chardin.
\newblock Implicitization using approximation complexes.
\newblock In {\em Algebraic geometry and geometric modeling}, Math. Vis., pages
  23--35. Springer, Berlin, 2006.

\bibitem{CFN08}
Marc Chardin, Amadou~Lamine Fall, and Uwe Nagel.
\newblock Bounds for the {C}astelnuovo-{M}umford regularity of modules.
\newblock {\em Math. Z.}, 258(1):69--80, 2008.

\bibitem{CLO98}
David Cox, John Little, and Donal O'Shea.
\newblock {\em Using algebraic geometry}, volume 185 of {\em Graduate Texts in
  Mathematics}.
\newblock Springer-Verlag, New York, 1998.

\bibitem{Cox05}
David~A. Cox.
\newblock Solving equations via algebras.
\newblock In {\em Solving polynomial equations}, volume~14 of {\em Algorithms
  Comput. Math.}, pages 63--123. Springer, Berlin, 2005.

\bibitem{Draisma2015}
Jan Draisma, Emil Horobe{\c{t}}, Giorgio Ottaviani, Bernd Sturmfels, and
  Rekha~R. Thomas.
\newblock The euclidean distance degree of an algebraic variety.
\newblock {\em Foundations of Computational Mathematics}, pages 1--51, 2015.

\bibitem{DREESEN20121203}
Philippe Dreesen, Kim Batselier, and Bart~De Moor.
\newblock Back to the roots: Polynomial system solving, linear algebra, systems
  theory.
\newblock {\em IFAC Proceedings Volumes}, 45(16):1203 -- 1208, 2012.
\newblock 16th IFAC Symposium on System Identification.

\bibitem{Eis95}
David Eisenbud.
\newblock {\em Commutative algebra}, volume 150 of {\em Graduate Texts in
  Mathematics}.
\newblock Springer-Verlag, New York, 1995.
\newblock With a view toward algebraic geometry.

\bibitem{M2}
Daniel~R. Grayson and Michael~E. Stillman.
\newblock Macaulay2, a software system for research in algebraic geometry.
\newblock Available at \url{http://www.math.uiuc.edu/Macaulay2/}.

\bibitem{HarrisBook}
Joe Harris.
\newblock {\em Algebraic geometry}, volume 133 of {\em Graduate Texts in
  Mathematics}.
\newblock Springer-Verlag, New York, 1995.
\newblock A first course, Corrected reprint of the 1992 original.

\bibitem{Hart77}
Robin Hartshorne.
\newblock {\em Algebraic geometry}.
\newblock Springer-Verlag, New York-Heidelberg, 1977.
\newblock Graduate Texts in Mathematics, No. 52.

\bibitem{HSV82}
J.~Herzog, A.~Simis, and W.~V. Vasconcelos.
\newblock Approximation complexes of blowing-up rings.
\newblock {\em J. Algebra}, 74(2):466--493, 1982.

\bibitem{HSV83}
J.~Herzog, A.~Simis, and W.~V. Vasconcelos.
\newblock Approximation complexes of blowing-up rings. {II}.
\newblock {\em J. Algebra}, 82(1):53--83, 1983.

\bibitem{JP14}
Alfrederic Josse and Fran{\c c}oise Pene.
\newblock Normal class and normal lines of algebraic surfaces.
\newblock Preprint arXiv:1402.7266, 2014.

\bibitem{KS14}
Kwanghee Ko and Takis Sakkalis.
\newblock Orthogonal projection of points in cad/cam applications: an overview.
\newblock {\em Journal of Computational Design and Engineering}, 1(2):116 --
  127, 2014.

\bibitem{krick2001}
Teresa Krick, Luis~Miguel Pardo, and Mart{\'\i}n Sombra.
\newblock Sharp estimates for the arithmetic nullstellensatz.
\newblock {\em Duke Math. J.}, 109(3):521--598, 09 2001.

\bibitem{La04}
Daniel Lazard.
\newblock Computing with parameterized varieties.
\newblock In {\em Algebraic geometry and geometric modeling}, Math. Vis., pages
  53--69. Springer, Berlin, 2006.

\bibitem{MC}
Dinesh Manocha and John Canny.
\newblock A new approach for surface intersection.
\newblock In {\em Proceedings of the first {ACM} symposium on Solid modeling
  foundations and {CAD/CAM} applications}, pages 209--219, Austin, Texas,
  United States, 1991. {ACM}.

\bibitem{PaMa02}
Nicholas~M. Patrikalakis and Takashi Maekawa.
\newblock {\em Shape interrogation for computer aided design and
  manufacturing}.
\newblock Springer-Verlag, Berlin, 2002.

\bibitem{NURBS}
Les~A. Piegl and Wayne Tiller.
\newblock {\em The NURBS book (2. ed.)}.
\newblock Monographs in visual communication. Springer, 1997.

\bibitem{PoLe03}
Helmut Pottmann and Stefan Leopoldseder.
\newblock A concept for parametric surface fitting which avoids the
  parametrization problem.
\newblock {\em Computer Aided Geometric Design}, 20(6):343 -- 362, 2003.

\bibitem{SC}
T.W. Sederberg and F.~Chen.
\newblock Implicitization using moving curves and surfaces.
\newblock In {\em Proceedings of SIGGRAPH}, volume~29, pages 301--308, 1995.

\bibitem{Shen:2016:LNS:3045889.3064444}
Jingjing Shen, Laurent Bus{\'e}, Pierre Alliez, and Neil Dodgson.
\newblock A line/trimmed nurbs surface intersection algorithm using matrix
  representations.
\newblock {\em Comput. Aided Geom. Des.}, 48(C):1--16, November 2016.

\bibitem{SJKW02}
Kyung-Ah Sohn, B.~Juttler, Myung-Soo Kim, and Wenping Wang.
\newblock Computing distances between surfaces using line geometry.
\newblock In {\em Computer Graphics and Applications, 2002. Proceedings. 10th
  Pacific Conference on}, pages 236--245, 2002.

\bibitem{TJD05}
Jan~B. Thomassen, Pal~H. Johansen, and Tor Dokken.
\newblock Closest points, moving surfaces; and algebraic geometry.
\newblock In {\em Mathematical methods for curves and surfaces~:~Tromsoe,
  2004}, pages 351--382. Nashboro Press, Brentwood, Tenn, 2005.

\bibitem{XiBuCi19}
Xiao Xiao, Laurent Bus{\'e}, and Fehmi Cirak.
\newblock A noniterative method for robustly computing the intersections
  between a line and a curve or surface.
\newblock {\em International Journal for Numerical Methods in Engineering},
  120(3):382--390, 2019.

\bibitem{XiSaCi19}
Xiao Xiao, Malcolm Sabin, and Fehmi Cirak.
\newblock Interrogation of spline surfaces with application to isogeometric
  design and analysis of lattice-skin structures.
\newblock {\em Computer Methods in Applied Mechanics and Engineering}, 351:928
  -- 950, 2019.

\end{thebibliography}
\end{document}